\documentclass[a4paper,12pt]{article}

\usepackage[english]{babel}
\usepackage{fancyhdr}
\usepackage{fullpage} 
\usepackage[dvips]{graphicx}
\usepackage[latin1]{inputenc}
\usepackage[T1]{fontenc}
\usepackage{fouriernc} 
\usepackage{hyperref} 

\usepackage{amssymb} 
\usepackage{amsmath} 
\usepackage{amsthm} 
\usepackage{array} 
\usepackage{color} 
\usepackage{amsfonts} 
\usepackage{mathrsfs}
\usepackage{amsxtra} 
\usepackage{mathtools}
\usepackage{enumerate}

\usepackage{scrextend} 

\numberwithin{equation}{section} 

\renewcommand{\epsilon}{\varepsilon}
\renewcommand{\phi}{\varphi}
\newcommand*{\de}{d}

\theoremstyle{plain} 
\newtheorem{theorem}{Theorem}[section]

\newtheorem{proposition}[theorem]{Proposition}
\newtheorem{lemma}[theorem]{Lemma}

\newtheorem{remark}[theorem]{Remark}

\newtheorem{assumption}[theorem]{Assumption}

\newtheorem{definition}[theorem]{Definition}

\newcommand{\leb}{L}
\newcommand{\lin}{L}
\newcommand{\ns}{\textrm{ }}

\usepackage{xcolor}
\hypersetup{
    colorlinks,
    linkcolor={green!60!blue},
    citecolor={green!60!blue},
    urlcolor={green!60!blue}
}

\usepackage{pifont}

  \def\Swiech
{{\accent1 S}wie{\hbox{\kern -0.21em\lower
0.77ex\hbox{$\textfont1=\scriptfont1
\lhook$}}}ch}

\def\SWIECH
{{\accent1 S}WIE{\hbox{\kern -0.21em\lower
0.77ex\hbox{$\textfont1=\scriptfont1
\lhook$}}}CH}

\title{Partial regularity of viscosity solutions for a class of Kolmogorov equations arising from mathematical finance} 
\date{} 
\author{M.\ Rosestolato\thanks{\thinspace
CMAP, \'Ecole Polytechnique, Paris, France,
e-mail: \texttt{mauro.rosestolato@polytechnique.edu}. 
This research has been partially supported by the INdAM-GNAMPA project ``Equazioni stocastiche con memoria e applicazioni'' (2014).}
 \and A.\ \Swiech\thanks{\thinspace
School of Mathematics, Georgia Institute of Technology, Atlanta, GA 30332, USA, e-mail: \texttt{swiech@math.gatech.edu}.
}
}

\begin{document}

\maketitle

\begin{abstract} 
We study value functions which are viscosity solutions of certain Kolmogorov equations. Using PDE techniques we prove that
they are $C^{1+\alpha}$ regular on special finite dimensional subspaces. The problem has origins in hedging 
 derivatives of risky assets in mathematical finance.
\end{abstract}

\vspace{10pt}
\noindent\textbf{Keywords:} viscosity solution, Kolmogorov equation, stochastic differential equation, delay problem,
 hedging problem.

\vspace{10pt} 
\noindent\textbf{AMS 2010 subject classification:} 35R15, 49L25, 60H15, 91G80.


\section{Introduction}

In this paper we study partial regularity of viscosity solutions for a class of Kolmogorov equations. Our motivation comes from mathematical finance,
more precisely from  hedging a
derivative of a risky asset whose volatility as well as the claim may depend on the past history of the asset. Our Kolmogorov equations are
thus associated to stochastic delay problems. They are linear second order partial differential equations in an infinite dimensional Hilbert space
with a drift term which contains an unbounded operator and a second order term which only depends on a finite dimensional component of the Hilbert
space. Such equations are typically investigated using the notion of the so-called $B$-continuous viscosity solutions (see 
\cite{FGS,Kelome02,'Swiech1994}). We impose conditions under which our Kolmogorov equations have unique $B$-continuous viscosity solutions. However general Hamilton-Jacobi-Bellman
equations associated to stochastic delay optimal control problems which are rewritten as optimal control problems for stochastic differential equations (SDEs) in an infinite dimensional Hilbert space are difficult, not well studied yet, and few results are available in the literature.

We work directly with the value function here since its partial regularity is of interest in the hedging problem and it is well known that 
under our assumptions the value function is the unique $B$-continuous viscosity solution of the Kolmogorov equation 
(see e.g.\ \cite{FGS,Kelome02}). We thus never
use the theory of $B$-continuous viscosity solutions. Instead our strategy for proving partial regularity of the value function is the following.
We consider SDEs with smoothed out coefficients and the unbounded operator replaced by its Yosida approximations and study the corresponding
value functions with smoothed out payoff function. The new value functions are G\^ateaux differentiable and converge
on compact sets to the original value functions. They also satisfy their associated Kolmogorov equations. We then prove that their finite dimensional
sections are viscosity solutions of certain linear finite dimensional parabolic equations for which we establish $C^{1,\alpha}$
estimates. Passing to the limit with the approximations, these estimates are preserved giving $C^{1,\alpha}$ partial regularity
for finite dimensional sections of the original value function.

Partial regularity results for first order unbounded HJB equations in Hilbert spaces associated to certain deterministic optimal control problems with 
delays have been obtained in \cite{FedGolGoz2010}. The technique of \cite{FedGolGoz2010} relied on arguments using concavity of the data and 
strict convexity of the Hamiltonian and provided 
$C^1$ regularity on one-dimensional sections corresponding to the so-called ``present'' variable. Here the equations are of second order,  we rely on approximations and parabolic regularity estimates, and we obtain regularity on $m$-dimensional sections. The reader can also consult \cite{Lions88} for various global and partial regularity results for bounded
HJB equations in Hilbert spaces (see also \cite{SwiechTeixeira09}).

We refer the reader to \cite{FGS,Lions88,Lions89jfa} for the theory of viscosity solutions for bounded second order HJB equations in Hilbert 
spaces and to \cite{FGS,Kelome02,'Swiech1994} for the theory of the so-called $B$-continuous viscosity solutions for unbounded second order HJB equations in Hilbert spaces. A fully nonlinear equation with a similar separated structure to our Kolmogorov equation \eqref{eq:2015-01-16:01aaa}
but with a nonlinear unbounded operator $A$ was studied in \cite{KocanSwiech95}. For classical results about Kolmogorov equation in Hilbert
spaces we refer the reader to \cite{DaPrato2004}.

The plan of the paper is the following. In the rest of the Introduction we explain the financial motivation of our problem. Section \ref{Sec:prelim}
contains notation and various results about mild solutions of the SDE, their extensions to a bigger space with a weaker topology
related to the original unbounded operator $A$, and various approximation results. In Section \ref{Sec:visc} we study viscosity solutions
of the approximating equations, investigate finite dimensional sections of viscosity solutions, and prove their regularity.

\subsection{Motivation from finance}
\label{sec:motivationfinance}

One motivation for the present study comes from 
the classical problem 
in
financial mathematics
of  hedging 
a
derivative of some risky assets.

Let us consider a financial market composed of
two assets: a risk free asset $P$ (a bond price), and a risky asset
$R$ (a stock price).  We assume that $P$ follows the deterministic
dynamics 
$
dP_s=rP_sds
$, where $r$ is the (constant) spot interest reate, and
that $R$ follows the dynamics
\begin{equation}
  \label{eq:2015-08-10:02}
  \begin{cases}
    dR_s=rR_sds+\nu(s,R_s)dW_s\qquad s\in(t,T],\\
    R_t= x ,
  \end{cases}
\end{equation}
where $(\Omega,\mathcal{F},\mathbb{F}=\{\mathcal{F}_t\}_{t\in[0,T]},\mathbb{P})$
is a filtered probability space, $T>0$ is the maturity date,  $ x \in
\mathbb{R}$, $t\in[0,T)$,
and
 $\nu$ satisfies the usual Lipschitz assumptions.
Denote by
$R^{t, x }$ the unique strong solution of SDE
\eqref{eq:2015-08-10:02}.

Given a function $\varphi\colon \mathbb{R}\rightarrow \mathbb{R}$, 
the problem of
hedging the derivative $\varphi(R^{0,x}_T)$
consists in
 finding a
 self-financing portfolio strategy 
replicating
$\varphi(R^{0, x }_T)$, i.e.\  a couple of real-valued processes
$\{(h_s^P,h_s^R)\}_{s\in[0,T]}$ such that the portfolio
$V_s\coloneqq h_s^PP_s+h_s^RR^{0,x}_s$, composed of $h_s^P$ shares of $P$ and 
$h_s^R$ shares of $R^{0,x}$, 
satisfies
\begin{equation}
  \label{eq:2016-05-18:01}
  \begin{dcases}
    dV_s=h^P_sdP_s+h^R_sdR_s^{0,x}&s\in[0,T)\\
    V_T=\varphi(R^{0, x }_T).&
  \end{dcases}
\end{equation}
The hedging problem can be solved as follows (see e.g.\ \cite[Ch.\ 8]{Bjork2003} for the financial argument and \cite[Ch.\ 7]{DaPrato2004} for the mathematical details).
We begin by introducing the function
\begin{equation}
  \label{eq:2015-08-09:02}
u( t, x )\coloneqq e^{-r(T-t)}\mathbb{E}\left[\varphi(R^{ t, x }_T)\right]\qquad\forall (t, x )\in [0,T]\times \mathbb{R}.
\end{equation}
Notice that, by Markov property of $R$,
we have
\begin{equation}
  \label{eq:2016-05-18:02}
  u(t,x)=e^{-rh}\mathbb{E} \left[ u(t+h,R^{t,x}_{t+h}) \right]\qquad 0\leq t,h,t+h\leq T.
\end{equation}
If 
$u(t,x)$
is Fr\'echet differentiable up to order $2$ with respect to $x$, with 
derivatives which are bounded and continuous 
 jointly in $(t,x)$,
then It\^o's formula and \eqref{eq:2016-05-18:02} permit to show 
that $u$ is actually $C^{1,2}$ and solves to
the following Kolmogorov-type partial differential equation
\begin{equation}\label{2015-08-09:00}
  \begin{cases}
    u_t+rxD_xu+\frac{1}{2}\nu^2(t,x)D^2_{x}u-ru=0&\qquad (t,x)\in (0,T)\times \mathbb{R},\\
    u(T,x)=\phi(x)&\qquad x\in \mathbb{R}.
  \end{cases}
\end{equation}
By
using \eqref{2015-08-09:00} and 
 applying 
It\^o's formula to $u(s,X^{0,x}_s)$,
we find 
the following 
representation formula
\begin{equation}
  \label{eq:2015-08-10:04}
 u(s,R^{0, x }_s)=u(0, x )+
  \int_0^sru(w,R^{0, x }_w)dw
  +
  \int_0^sD_{ x }u(w,R^{0, x }_w)\nu(w,R^{0, x }_w)dW_w.
\end{equation}
Finally, by recalling the definition of $u$
and considering formula \eqref{eq:2015-08-10:04},
we can see that 
the  portfolio strategy
\begin{equation}
  \label{eq:2015-08-09:01}
  h^P_s=\frac{u(s,R^{0, x }_s)-D_{ x }u(s,R^{0, x }_s)R^{0, x }_s}{P_s}\qquad\mbox{and}\qquad
  h^R_s=D_{ x }u(s,R^{0, x }_s)\qquad \forall s\in[0,T),
\end{equation}
solves the hedging problem.
Indeed, 
we have 
$$
V_s\coloneqq 
h_s^PP_s+h_s^RR^{0,x}_s= u(s,X^{0,x}_s) \qquad \forall s\in[0,T],
$$
hence in particular $    V_T=u(T,X_T^{0,x})=\varphi(R^{0,x}_T)$.
Moreover, by \eqref{eq:2015-08-10:04}, we have the self-financing condition
\begin{equation*}
  \label{eq:2016-05-18:03}
    dV_s=h^P_sdP_s+h_s^{R}dR^{0,x}_s \qquad \forall s\in[0,T).
\end{equation*}

\bigskip
There are three essential features
of the model that allow to implement the program above:
\begin{enumerate}[(1)]
\item The Markov property of $R$, which makes 
\eqref{eq:2016-05-18:02}
possible.
\item The existence of $D_{ x }u$, which lets the portfolio strategy be defined by
  \eqref{eq:2015-08-09:01}
\item\label{2015-11-19:00} The availability of It\^o's formula and the fact that $u$ solves to \eqref{2015-08-09:00}, in order to derive \eqref{eq:2015-08-10:04}, hence to see that 
\eqref{eq:2015-08-09:01}
is the hedging strategy.
\end{enumerate}

\smallskip 
Let us now consider a slightly more general risky asset
$R$, in which the volatility depends not only on the value $R_s$ of
$R$ at time $s$, but also on the entire past values of $R$. That is,
the dynamics of $R$ has the following form
\begin{equation}
  \label{eq:2015-08-10:01}
  \begin{dcases}
    dR_s=rR_sds+\nu(s,R_s,\{R_{s+s'}\}_{s'\in (-\infty,0)})dW_s&\qquad s\in(t,T]\\
    R_t=x_{0}&\\
    R_{t'}=x_1(t')&\qquad
    t'\in(-\infty,t), 
  \end{dcases}
\end{equation}
where $x_0\in\mathbb{R}$ and $x_1\colon (-\infty,0)\rightarrow
\mathbb{R}$ is a given deterministic funtion belonging to
$L^2(\mathbb{R}^-,\mathbb{R})$, expressing the past history of the
stock price $R$ up to time $t$.  We also would like to face the case
in which the European claim depends itself on the history of $R$, 
i.e.\  it has the form $\varphi(R_T,\{R_t\}_{t\in(-\infty,T)})$.

We point out that model \eqref{eq:2015-08-10:01} can 
also include the case in which the path-dependency is only relative to a 
finite past window $[-d,0]$, i.e.\ $\nu$ 
is defined
as a function of the past history of $R$ only from the past date $t-d$ up to the present $t$.
To fit this case into 
\eqref{eq:2015-08-10:01}, it is sufficient to replace the coefficient $\nu$ in \eqref{eq:2015-08-10:01}
 by a $\nu'$ defined by
$$
\nu'(s,R_s,\{R_{s+s'}\}_{s'\in(-\infty,0)})
\coloneqq
\nu(s,R_s,\mathbf{1}_{[-d,0)}(\cdot)\{R_{s+s'}\}_{s'\in(-\infty,0)}).
$$
In such a case,  it is easily seen that $R$  does not depend on the tail $\mathbf{1}_{(-\infty,-d)}(\cdot)x_1$ of the initial datum.
Hence a delay model with a  finite  delay window  can be rewritten in the form \eqref{eq:2015-08-10:01}.

A natural question is if we
can
solve the hedging problem for the delay case
by implementing the standard arguments outlined above
for the case in which $R$ is given by \eqref{eq:2015-08-10:02}.
We now see that this can be done, if we take into account
the three features mentioned above which make the machinery work.

If $R^{t,(x_0,x_1)}$ solves
\eqref{eq:2015-08-10:01}, then in general it is not Markovian.
Moreover,
since both the claim $\varphi$ and 
 the function $u$, now defined by
$$
u(t,x_0,x_1)\coloneqq e^{-r(T-t)}\mathbb{E} \left[ \varphi\big(R^{t,x_0,x_1}_T, \big\{ R^{t,x_0,x_1}_{t'} \big\}_{t'\in(-\infty,T)}\big)  \right]\quad \forall (t,x_0,x_1)\in [0,T]\times \mathbb{R}\times L^2(\mathbb{R}^-,\mathbb{R}),
$$
are path-dependent, the analogous PDE \eqref{2015-08-09:00} would now be path-dependent, and it would be necessary to employ a stochastic calculus for 
path-dependent functionals of
It\^o processes in order to relate $u$ with the PDE, as done for the non-path-dependent case.

A classical  workaround tool to
regain Markovianity and
avoid the complications of a path-dependent stochastic calculus consists in
rephrasing the model in a functional space setting. 
What we lose
 by doing so
is that the dynamics will evolve in an infinite
dimensional space.
 We briefly recall how the rephrasing works.
We refer the reader to
 \cite{Chojnowska-Michalik1978} for the case with finite delay. 
The argument extends without difficulty to the case with infinite delay.

We first introduce the Hilbert space $H\coloneqq \mathbb{R}\times L^2(\mathbb{R}^-,\mathbb{R})$, the functions
\begin{equation}\label{eq:2015-08-10:06}
  \begin{split}
    &F\colon [0,T]\times H \rightarrow H,\ 
    (x_0,x_1)\mapsto \left(rx_0,0\right)\\
    &\Sigma\colon[0,T]\times H \rightarrow H,\ 
     (x_0,x_1)\mapsto \left(\nu(t,x_0,x_1),0\right),
\end{split}
\end{equation}
and the strongly continuous semigroup of translations on $H$,
i.e.\ the family $\hat S\coloneqq \{\hat S_t\}_{t\in\mathbb{R}^+}$ of linear continuous
operators defined by
$$
\hat S_t\colon H\to H,\ (x_0,x_1)\mapsto
(x_0,x_1(t+\cdot)\mathbf{1}_{(-\infty,-t)}(\cdot)+x_0\mathbf{1}_{[-t,0]}(\cdot)).
$$
The infinitesimal generator $\hat A$ of $\hat S$ is given by
$$
\hat A\colon D(\hat A)\rightarrow H,\ (x_0,x_1)\mapsto (0,x_1'),
$$
where
$$
D(\hat A)=\left\{(x_0,x_1)\in H\colon x_1\in
  W^{1,2}(\mathbb{R}^-)\ns\mathrm{and}\ x_0=x_1(0)\right\}.
$$
Then we consider the $H$-valued dynamics
\begin{equation}\label{eq:2015-08-10:00}
  \begin{dcases}
        {\displaystyle d \hat X_s} = \left(\hat A\hat X_s + F\left(s,\hat X_s\right)\right) d s+ \Sigma \left(s,\hat X_s\right) d W_s & s\in(t,T],\\
    \hat X_t = (x_0,x_1),
  \end{dcases}
\end{equation}
where $(x_0,x_1)\in H$, $t\in[0,T)$.  Under usual Lipschitz assumptions on
$\nu$, it can be shown that \eqref{eq:2015-08-10:00} has a unique mild
solution $\hat X^{t,(x_0,x_1)}$ (we refer to \cite{DaPrato2014} for stochastic differential equations in Hilbert spaces).  The link between \eqref{eq:2015-08-10:01}
and \eqref{eq:2015-08-10:00} is given by the following equation:
\begin{equation}
  \label{eq:2015-08-10:05}
 \mbox{for all $s\in[t,T]$},\ \hat X^{t,(x_0,x_1)}_s=\big(R_s^{t,(x_0,x_1)},
    \big\{
      R_{s'+s}^{t,(x_0,x_1)}
    \big\}_{s'\in(-\infty,0)}
  \big)\ \mbox{$\mathbb{P}$-a.s.,}
\end{equation}
where $R^{t,(x_0,x_1)}$ denotes the unique strong solution of
\eqref{eq:2015-08-10:01}. Observe that $\hat X$ is Markovian and no path-dependency appears in the coefficients $F$, $\Sigma$. This is the natural rephrasing of the dynamics of $R$ to get a Markovian setting for which the basic tools of stochastic calculus in Hilbert spaces (such as It\^o's formula) are available.

\smallskip We need an additional step to let the model studied in the paper apply to the financial problem we are considering. 
We rephrase \eqref{eq:2015-08-10:00} as an SDE in the same Hilbert space $H$, but with a maximal dissipative unbounded operator. 
To this goal, we observe that $A\coloneqq \hat A-\frac{1}{2}$ is a maximal dissipative operator generating the semigroup of contractions $S\coloneqq \{S_t\coloneqq e^{- t/2}\hat S_t\}_{t\in \mathbb{R}^+}$. Let us define $G(t,x)\coloneqq F(t,x)+\frac{x}{2}$,  $(t,x)\in[0,T]\times H$. Denote by $X^{t,(x_0,x_1)}$
 the unique mild solution of the SDE
\begin{equation}\label{2015-10-27:00}
  \begin{dcases}
 d X_s = \left( AX_s + G\left(s,X_s\right)\right) d s+ \Sigma \left(s,X_s\right) d W_s & s\in(t, T],\\
    X_t = (x_0,x_1).
  \end{dcases}
\end{equation}
It is not difficult to see that $\hat X^{t,(x_0,x_1)}=X^{t,(x_0,x_1)}$. Indeed, if $\{\hat A_\lambda\}_{\lambda>1/2}$ denote the Yosida approximations of 
$\hat A$, then the strong solution of
\begin{equation}\label{2015-10-27:01}
  \begin{dcases}
     d  X_{\lambda,s} = \big( \hat A_\lambda  X_{\lambda,s} + F\big(s, X_{\lambda,s}\big)\big) d s+ \Sigma \big(s, X_{\lambda,s}\big) d W_s & s\in (t,T],\\
     X_{\lambda,t} = (x_0,x_1),
   \end{dcases}
 \end{equation}
coincides with the strong solution $X_\lambda^{t,(x_0,x_1)}$ of
\begin{equation}\label{2015-10-27:02}
  \begin{dcases}
      X_{\lambda,s}= \left(  \left( \hat A_\lambda -\frac{1}{2} \right)   X_{\lambda,s} + G\left(s, X_{\lambda,s}\right)\right) d s+ \Sigma \left(s, X_{\lambda,s}\right) d W_s & s\in (t,T],\\
    X_{\lambda,t} = (x_0,x_1),
  \end{dcases}
\end{equation}
by the very definition and by uniqueness of strong solutions. Recalling that strong and mild solutions coincide when the linear 
 operator appearing in the drift is bounded\footnote{\label{2016-05-18:00}
%
   This can be seen by an easy application of Ito's formula, together
 with uniqueness of mild solutions.},
 $X_{\lambda}^{t,(x_0,x_1)}$ solves  \eqref{2015-10-27:02} in the mild sense. Now observe that $\hat A_\lambda-\frac{1}{2}$ generates the semigroup $\hat S_\lambda\coloneqq\{ \hat S_{\lambda,t}\coloneqq e^{-t/2}e^{\hat A_\lambda t}\}_{t\in \mathbb{R}^+}$. Since $e^{\hat A_\lambda t}\to \hat S_t$ strongly as $\lambda\rightarrow +\infty$, we have also $\hat S_{\lambda,t}\to S_t$ strongly. Then the mild solution $ X_{\lambda}^{t,(x_0,x_1)}$ converges to the mild solution $X^{t,(x_0,x_1)}$ as $\lambda\to +\infty$ (see e.g.\ the argument used to show Proposition \ref{prop:2015-06-20:03}-\emph{(\ref{2015-09-28:08})}).
Similarly, $X_\lambda^{t,(x_0,x_1)}$ solves
\eqref{2015-10-27:01} in the mild sense and then $X_\lambda^{t,(x_0,x_1)}\to \hat X^{t,(x_0,x_1)}$ as $\lambda\to +\infty$. We thus conclude that
$\hat X^{t,(x_0,x_1)}=X^{t,(x_0,x_1)}$ in a suitable space of processes where the well-posedness of the SDEs and the convergences above are considered.

It follows that equation \eqref{eq:2015-08-10:05} can be rewritten as:
\begin{equation}
  \label{2015-10-27:03}
 \mbox{for all $s\in[t,T]$,}\;\;\;   X^{t,(x_0,x_1)}_s=\big(R_s^{t,(x_0,x_1)},
    \big\{
      R_{s'+s}^{t,(x_0,x_1)}
    \big\}_{s'\in(-\infty,0)}
  \big)\;\;\; \mbox{$\mathbb{P}$-a.s.}.
\end{equation}
Having  \eqref{2015-10-27:03}, the
function $u$ can be written as
\begin{equation}
  \label{eq:2015-08-10:07}
  u( t,x_0,x_1)=e^{-r(T-t)}\mathbb{E}\left[\varphi\big(X^{ t,(x_0,x_1)}_T\big)\right]\qquad\forall (t,(x_0,x_1))\in [0,T]\times H.
\end{equation}
Thanks to the special structure of $\Sigma$ in SDE
\eqref{2015-10-27:00}, if $u$ has enough regularity to perform the
computations, it turns out that, for $s\in[0,T]$,
\begin{equation}
  \label{2015-08-10:05}
  \begin{multlined}[c][0.85\displaywidth]
        u\big(s,X^{0,(x_0,x_1)}_s \big)=u\big(0,(x_0,x_1)\big)+
    \int_0^sru \big(w,X^{0,(x_0,x_1)}_w \big)dw\\
    +\int_0^sD_{x_0}u \big(w,X^{0,(x_0,x_1)}_w \big)\nu \big(
      w,X^{0,(x_0,x_1)}_w \big) dW_w,
    \end{multlined}
  \end{equation}
and the only derivative of $u$ appearing in the above formula is the
directional derivative $D_{x_0}u$ with respect to the variable $x_0$,
representing the ``present'', according to the rephrasing $R\leadsto X$.
Once \eqref{2015-08-10:05} is available, one can verify, as it is done for the case without delay,  that
\begin{equation*}
  h^P_s=\frac{u\big(s,X^{0, (x_0,x_1) }_s\big)-D_{ x_0 }u\big(s,X^{0, (x_0,x_1) }_s\big)X^{0, (x_0,x_1) }_{0,s}}{P_s}\;\;\;\;\mbox{and}\;\;\;\;
  h^R_s=D_{ x_0 }u\big(s,X^{0, (x_0,x_1) }_s\big)\ \  \forall s\in[0,T)
\end{equation*}
solve the hedging problem in the delay case.

\smallskip The goal of this paper is to show the regularity of the
function $u$, defined by \eqref{eq:2015-08-10:07}, with respect to the
component $x_0$, when all the data are assumed to be Lipschitz with
respect to a particular norm associated to the operator $A$.

\bigskip
{\bf Acknowledgments.} The authors are grateful to the anonymous referee for valuable comments.

\section{Preliminaries}\label{Sec:prelim}

\subsection{Notation}
Let $\left(\Omega,\mathcal F, \mathbb{P}\right)$ be a complete probability space, let $T>0$, and let 
$\mathbb{F}=\left\{\mathcal F_t\right\}_{t\in [0,T]}$ be a 
filtration on $\left(\Omega,\mathcal F,\mathbb{P}\right)$ satisfying the usual conditions.
Define $\Omega_T=\Omega\times[0,T]$.  Denote by $\mathcal P$ the
$\sigma$-algebra in $\Omega_T$ generated by the sets $A_s\times
(s,t]$, where $A_s\in\mathcal F_s$, $0\leq s<t\leq T$, and 
$A_0\times \{0\}$, where $A_0\in\mathcal F_0$.  An element of $\mathcal P$ is
called a predictable set. We denote $\mathbb{R}^-=(-\infty$,0], $\mathbb{R}^+=[0,+\infty)$.

Let $(F,|\cdot|_F)$\footnote{We use the same symbol $|\cdot|$ to denote the norm of a normed space when the space is clear 
from the context. If not, we will clarify the space of reference with a subscript.} be a real separable Banach space. We define the following spaces:
\newcounter{saveenum}
  \begin{enumerate}[(i)]
  \item For $p\geq 1$, $L_\mathcal{P}^p(F)\coloneqq L^p_\mathcal{P}(\Omega_T,F)$ is the Banach space of $F$-valued predictable processes 
  $X$ such that
$$
|X|_{L_\mathcal{P}^p(F)}\coloneqq \left(\mathbb{E}\left[\int_0^T|X_t|_F^pdt\right]\right)^{1/p}<+\infty.
$$
\item  $\mathcal{H}^p_\mathcal{P}(F)$ is the subspace of elements $X$ of $L^p_\mathcal{P}(F)$ such that
$$
|X|_{\mathcal{H}^p_\mathcal{P}(F)}\coloneqq
\sup_{t\in [0,T]}\left(\mathbb{E}\left[|X_t|_F^p\right]\right)^{1/p}<+\infty,
$$
and, for all $t'\in[0,T]$, 
$$
\lim_{t\rightarrow t'}\mathbb{E} \left[ |X_t-X_{t'}|_F^p \right] =0.
$$
$\mathcal{H}^p_\mathcal{P}(F)$, when endowed with the norm $|\cdot|_{\mathcal{H}^p_\mathcal{P}(F)}$, is a Banach space.

We will consider $\mathcal{H}^p_\mathcal{P}(F)$ also with other norms. For $\gamma>0$, define
$$
|X|_{\mathcal{H}^p_\mathcal{P}(F),\gamma}\coloneqq
\sup_{t\in [0,T]} \left( e^{-\gamma t}\left(\mathbb{E}\left[|X_t|_F^p\right]\right)^{1/p} \right) .
$$
The norms $|\cdot|_{\mathcal{H}^p_\mathcal{P}(F)}$
and
$|\cdot|_{\mathcal{H}^p_\mathcal{P}(F),\gamma}$ are equivalent.
\setcounter{saveenum}{\value{enumi}}
\end{enumerate}
Let $n\geq 0$, $k\geq 0$, $T>0$, and let $E$, $F$ be real separable Banach spaces.
\begin{enumerate}[(i)]
\setcounter{enumi}{\value{saveenum}}
  \item $\mathcal{G}^1_s(E,F)$ denotes the space of continuous functions $f\colon E\to F$ such that the G\^ateaux derivative 
  $\nabla f(x)$ exists for every $x\in E$, the function
$$
\nabla f\colon E\rightarrow L(E,F)
$$
is strongly continuous and
$$
\sup_{x\in E}|\nabla f(x)|_{L(E,F)}<+\infty.
$$
When $E$ is a Hilbert space and $F=\mathbb{R}$, we will identify $\nabla f $ with  an element of $E$ through the Riesz representation $E^*=E$.
 \item  $\mathcal{G}^{0,1}_s([0,T]\times E,F)$ denotes the space of continuous functions $f\colon [0,T]\times E\to F$, 
such that the G\^ateaux derivative in the $x$ variable
  $\nabla_x f(t,x)$ exists for every $x\in E$, the function
$$
\nabla_x f\colon [0,T]\times E\rightarrow L(E,F)
$$
is strongly continuous and
$$
\sup_{(t,x)\in [0,T]\times E}|\nabla_x f(t,x)|_{L(E,F)}<+\infty.
$$
\item $C_b^{1}( E,F)$  denotes the space of continuous functions $f\colon E\rightarrow F$, continuously Fr\'echet differentiable, and such that
$$
\sup_{x\in  E}|Df(x)|_{L(E,F)}<+\infty,
$$
where $Df$ denotes the Fr\'echet derivative of $f$.
\item  $C^{0,1}([0,T]\times E,F)$  denotes the space of continuous functions $f\colon [0,T]\times E\rightarrow F$, continuously Fr\'echet differentiable with respect to the second variable.
\item
  $C_b^{0,1}([0,T]\times E,F)$  denotes the space of
functions $f\in C^{0,1}([0,T]\times E,F)$
 such that
$$
\sup_{(t,x)\in [0,T]\times E}|D_xf(t,x)|_{L(E,F)}<+\infty,
$$
where $D_xf$ denotes the Fr\'echet derivative of $f$ with respect to $x$.
\setcounter{saveenum}{\value{enumi}}
\end{enumerate}
When $F=\mathbb{R}$, we drop $\mathbb{R}$ and simply write $L^p_\mathcal{P}$, $\mathcal{H}^p_\mathcal{P}$, $\mathcal{G}^{1}_s(E)$, $\mathcal{G}^{0,1}_s(E),C_b^{0,1}([0,T]\times E)$, and $C_b^{0,1}([0,T]\times E)$. 

Though the notation could appear to be misleading, observe that
if $f\in C^{0,1}_b([0,T]\times E,F)$ or $f\in C^1_b(E,F)$, then $f$ is not supposed to be bounded.

\vskip5pt
Let  $m>0$ be a positive integer,
and let $U$ be an open subset of $\mathbb{R}^m$. Let $a,b$ be real numbers such that $a<b$. 
Define $Q\coloneqq [a,b)\times U$ and $\partial_P Q\coloneqq [a,b]\times \partial U\cup \{b\}\times U$.
\begin{enumerate}[(i)]
\setcounter{enumi}{\value{saveenum}}
\item  For $\alpha\in (0,1)$, $C^{1+\alpha}(Q)$ denotes the space of continuous functions $f\colon Q\rightarrow \mathbb{R}$ such that $D_xf(t,x)$ exists classically for every $(t,x)\in Q$, and such that
$$
|f|_{C^{1+\alpha}(Q)}\coloneqq |f|_\infty+|D_xf|_\infty+\sup_{\substack{(t,x),(s,y)\in Q\\(t,x)\neq (s,y)}}\frac{|u(s,y)-u(t,x)-\langle D_xf(t,x),y-x\rangle_m|}{\left(|t-s|+|x-y|_m^2\right)^{(1+\alpha)/2}}<+\infty,
$$
where 
$|\cdot|_\infty$ is the supremum norm, and
$|\cdot|_m$ and $\langle\cdot,\cdot\rangle_m$ are the Euclidean norm and scalar product in $\mathbb{R}^m$ respectively.
\item  For $\alpha\in (0,1)$, $C_{\rm loc}^{1+\alpha}((0,T)\times  \mathbb{R}^m)$ denotes the space of continuous functions $f\colon (0,T)\times \mathbb{R}^m\rightarrow \mathbb{R}$ such that, for every point $(t,x)\in (0,T)\times \mathbb{R}^m$, there exists $\epsilon >0$ and $a,b\in (0,T)$, with $a< b$, such that $f\in C^{1+\alpha}([a,b)\times B(x,\epsilon))$\footnote{$B(x,\epsilon)$ denotes the open ball centered at $x$ of radius 
$\epsilon$.}.
\item  For $p\geq 1$, $W^{1,2,p}(Q)$ denotes the usual Sobolev space of functions $f\in L^p(Q)$, whose weak
partial derivatives $u_t$, $f_{x_i}$ and $f_{x_ix_j}$ belong to $\leb^p(Q)$. $W^{1,2,p}(Q)$ is equipped with the norm
$$
|f|_{W^{1,2,p}(Q)}\coloneqq \left(|f|^p_{\leb^p(Q)}+|f_t|^p_{\leb^p(Q)}+|D_xf|^p_{\leb^p(Q)}+|D^2_{x}f|^p_{\leb^p(Q)}\right)^{1/p}.
$$
\end{enumerate}

\subsection{$H_B$-extensions of mild solutions of SDEs}\label{s:2015-01-20:11}

Let $m\geq 1$, and let $H_1$ be a real separable Hilbert space with scalar product $\langle\cdot,\cdot\rangle_{H_1}$.
Define $H\coloneqq\mathbb{R}^m\times H_1$.
Whenever $x$ is a point of $H$, we will denote by $x_0$ the component of $x$ in $\mathbb{R}^m$ and by $x_1$ the  component of $x$ in $H_1$.
We endow $H$ with the natural scalar product
$$
\langle (x_0,x_1),(y_0,y_1)\rangle\coloneqq \langle x_0,y_0\rangle_m+\langle x_1,y_1\rangle_{H_1}\qquad \forall  (x_0,x_1),(y_0,y_1)\in H.
$$

 Let
$G\colon [0,T]\times H\to H$ and $\sigma\colon [0,T]\times H\rightarrow L(\mathbb{R}^m)$.
We will consider the following assumptions on them.
 \begin{assumption}\label{ass:2015-06-04:01}
The functions $G$ and $\sigma$ are continuous, and there exists $M>0$ such that
$$
| G(t,x)- G(t,y)|_H
+
|\sigma(t,x)-\sigma(t,y)|_{L(\mathbb{R}^m)}\leq M|x-y|_H\qquad \forall (t,x),\,(t,y)\in [0,T]\times H.
$$
\end{assumption}
\vskip10pt
We associate to $\sigma$ the following function:
$$  \Sigma\colon [0,T]\times H\rightarrow L(\mathbb{R}^m,H),$$
defined by
\begin{equation}
  \label{eq:2015-06-20:01}
  \Sigma(t,x)y=(\sigma(t,x)y,0_1)
\end{equation}
for $(t,x)\in [0,T]\times H$, $y\in \mathbb{R}^m$, and where $0_1$ denotes the origin in $H_1$.

\vskip10pt
The following assumption will be standing for the remaining part of the work.

\begin{assumption}\label{ass:2015-06-17:00}
  $S$ is a strongly continuous semigroup of contractions, with $A$ as its infinitesimal generator. 
\end{assumption}

 We remark that Assumption \ref{ass:2015-06-17:00} implies that $A$ is a linear densely defined maximal dissipative operator on $H$.
In the rest of the paper $A$ is an abstract operator which may be different from the operator $A$ introduced in Section \ref{sec:motivationfinance}.

Let $W$ be a standard $m$-dimensional Brownian motion with respect  to the filtration $\mathbb{F}$.
For $t\in [0,T)$ and $x\in H$, consider the SDE
\begin{equation}\label{eq:2013-02-19:ab}
  \begin{dcases}
  d X_s = \left(AX_s +  G\left(s,X_s\right)\right) \de s+ \Sigma\left(s,X_s\right) \de W_s \qquad s\in(t,T]\\
    X_t =x.
  \end{dcases}
\end{equation}
It is well known (see \cite[Ch.\ 7]{DaPrato2014}) that, under Assumption \ref{ass:2015-06-04:01}, 
for $p\geq 2$, there exists a unique mild solution in $\mathcal{H}^p_\mathcal{P}(H)$ to \eqref{eq:2013-02-19:ab}, i.e.\ a unique process $X^{t,x}\in \mathcal{H}^p_\mathcal{P}(H)$ such that
\begin{equation*}
  X^{t,x}_s=
\begin{dcases}
  x&\qquad  s\in[0,t]\\
  S_{s-t}x+\int_t^s S_{s-w} G(w,X_w^{t,x}) dw +\int_t^s S_{s-w}\Sigma(w,X_w^{t,x}) dW_w
&\qquad s\in(t,T].
\end{dcases}
\end{equation*}
Moreover, for every $t\in[0,T]$, the map 
\begin{equation}
  \label{eq:2015-06-19:00}
  H\to \mathcal H^p_\mathcal{P}(H),\quad x\mapsto X^{t,x}
\end{equation}
is continuous and Lipschitz. 

For future reference,
we state existence and uniqueness of mild solution in the following proposition, where we also show continuity in $t$, and we introduce tools useful for later proofs.

\begin{proposition}\label{2015-09-27:11}
  For any $p\geq 2$, under Assumption \ref{ass:2015-06-04:01}, there exists a unique mild solution
  $X^{t,x}\in \mathcal{H}^p_\mathcal{P}(H)$ to SDE \eqref{eq:2013-02-19:ab}, and the map
  \begin{equation}\label{2015-09-27:00}
    [0,T]\times (H,|\cdot|)\rightarrow \mathcal{H}_\mathcal{P}^p(H),\,(t,x)\mapsto X^{t,x}
  \end{equation}
  is continuous in $(t, x )$, and Lipschitz in $x$, uniformly in
  $t$. 
\end{proposition}
\begin{proof}
Since the arguments are standard, we just give a sketch of proof.
Let $t\in[0,T]$.
Define the map
$$
\Phi(t;\cdot,\cdot)\colon H\times \mathcal{H}^p_\mathcal{P}(H)\rightarrow \mathcal{H}^p_\mathcal{P}(H)
$$
by
$$
\Phi(t;x,Z)_s\coloneqq 
  \begin{dcases}
  x&s\in[0,t)\\
  S_{s-t}x+\int_t^sS_{s-w}
G(w,Z_w) dw +\int_t^s S_{s-w}\Sigma(w,Z_w) dW_w&s\in[t,T].
\end{dcases}
$$
Let $\gamma>0$. By Assumption \ref{ass:2015-06-04:01}, we have
\begin{gather}
\sup_{t\in[0,T]}e^{-\gamma pt}  \mathbb{E} \left[ |G(t,Z_t)-G(t,Z'_t)|^p \right]
 \leq 
M^p|Z-Z'|_{\mathcal{H}^p_\mathcal{P}(H),\gamma}^p \ \ \ \forall Z,Z'\in \mathcal{H}^p_\mathcal{P}(H)
\label{2015-09-27:02}\\
\sup_{t\in[0,T]}e^{-\gamma pt}  \mathbb{E} \left[ |\sigma(t,Z_t)-\sigma(t,Z'_t)|^p_{L(\mathbb{R}^m)} \right]
 \leq 
M^p|Z-Z'|_{\mathcal{H}^p_\mathcal{P}(H),\gamma}^p \ \ \ \forall Z,Z'\in \mathcal{H}^p_\mathcal{P}(H).\label{2015-09-27:04}
\end{gather}
By \eqref{2015-09-27:02}, \eqref{2015-09-27:04}, the linearity of $\Phi(t;x,Z)$ in $x$, and
 \cite[Ch.\ 7, Proposition~7.3.1]{DaPrato2004},
there exists $\gamma>0$, depending only on 
$p$, $T$, $M$,
 such that 
 \begin{equation}
   \label{eq:2015-09-27:05}
   \sup_{(t,x)\in[0,T]\times H}|\Phi(t;x,Z)-\Phi(t;x,Z')|_{\mathcal{H}^p_\mathcal{P}(H),\gamma}\leq \frac{1}{2}
|Z-Z'|_{\mathcal{H}^p_\mathcal{P}(H),\gamma} \ \ \ \forall Z,Z'\in \mathcal{H}^p_\mathcal{P}(H).
\end{equation}
This shows that, for every $(t,x)\in[0,T]\times H$, there exists a unique fixed point $X^{t,x}\in \mathcal{H}^p_\mathcal{P}(H)$ of $\Phi(t;x,\cdot)$. Such a fixed point is the mild solution of \eqref{eq:2013-02-19:ab}. 

The continuity of \eqref{2015-09-27:00} is also standard. We sketch a slightly different argument.
Let $\{t_n\}_{n\in \mathbb{N}}$ be a sequence converging to $t$ in $[0,T]$.
By standard estimates on the integrals defining $\Phi$  (for the stochastic integral using Burkholder-Davis-Gundy's inequality), by sublinear growth of $G(t,x)$ and $\sigma(t,x)$ in $x$ uniformly in $t$, and by Lebesgue's dominated convergence theorem,
we have
\begin{equation}
  \label{eq:2015-09-27:07}
  \lim_{n\rightarrow +\infty}\Phi(t_n;x,Z)=\Phi(t;x,Z)\ \mathrm{in}\ \mathcal{H}_\mathcal{P}^p(H),\ \forall (x,Z)\in H\times \mathcal{H}^p_\mathcal{P}(H).
\end{equation}
Then, by \eqref{eq:2015-09-27:07},  \eqref{eq:2015-09-27:05}, and \cite[Theorem 7.1.5]{DaPrato2004}, we have 
\begin{equation}
  \label{eq:2015-09-27:10}
  \lim_{n\rightarrow +\infty} X^{t_n,x}=X^{t,x}\ \ \ \mathrm{in}\ \mathcal{H}_\mathcal{P}^p(H),\ \forall x\in H.
\end{equation}
This shows the continuity in $t$ of $X^{t,x}$.
We notice that  
 \begin{equation}
   \label{2015-09-27:08}
   \sup_{(t,Z)\in[0,T]\times \mathcal{H}^p_\mathcal{P}(H)}|\Phi(t;x,Z)-\Phi(t;x',Z)|_{\mathcal{H}^p_\mathcal{P}(H),\gamma}\leq
|x-x'|_{H} \ \ \ \forall x,x'\in H.
\end{equation}
By applying \cite[inequality ($***$) on p.\ 13]{Granas2003},
 we obtain
\begin{equation}
  \label{eq:2015-09-27:09}
  \sup_{t\in[0,T]}|X^{t,x}-X^{t,x'}|_{\mathcal{H}^p_\mathcal{P}(H),\gamma}\leq 2|x-x'|_H\ \ \ \forall x,x'\in H.
\end{equation}
By \eqref{eq:2015-09-27:10} and by \eqref{eq:2015-09-27:09} we conclude that the map
$$
[0,T]\times H\rightarrow \mathcal{H}^p_\mathcal{P}(H),\ (t,x) \mapsto X^{t,x}
$$
is continuous and Lipschitz continuous in $x$ uniformly in $t$.
\end{proof}


We are going to endow $H$ with a weaker
norm, and give conditions such that the above continuity in $(t,x)$
  of $X^{t,x}$ 
 extends to 
the new norm. We will also make assumptions which will guarantee the G\^ateaux differentiability of the mild solution with respect to the initial datum $x$ in the space with the weaker norm and the strong continuity of the the G\^ateaux derivative.

\vskip10pt
Let $R\colon D(R)\rightarrow H$ be a densely defined linear operator such that $R\colon  D(R)\to H$ has inverse
$R^{-1}\in \lin(H)$.  Then $B=\left(R^*\right)^{-1}R^{-1}\in \lin(H)$
is selfadjoint and positive. For $x\in H$, define
\begin{equation}
  \label{eq:2016-05-19:00}
  | x|^2_B=\langle Bx,x\rangle=\left| R^{-1}x\right|^2_H
\end{equation}
Such norms have been introduced in the context of the so-called $B$-continuous viscosity solutions of HJB equations in
\cite{Crandall1990237,CRANDALL1991417}
 and used in many later works on HJB equations in infinite dimensional spaces (see \cite[Ch.\ 3]{FGS} for more on this).
The space $H$ endowed with the norm $| \cdot |_B$ is
pre-Hilbert, since $| \cdot |_B $ is inherited by the scalar
product $\langle x,y\rangle_B=\langle B^{1/2}x,B^{1/2}y\rangle$, where
$B^{1/2}$ is the unique positive self-adjoint continuous linear
operator such that $B=B^{1/2}B^{1/2}$.  Denote by $H_B$ the completion of
the pre-Hilbert space $\left(H,| \cdot |_B\right)$. With some abuse
of notation, we also denote by $| \cdot|_B$ the extension of 
$|\cdot |_B$ to $H_B$.  

By definition of $|\cdot|_B$, $R\colon  (D(R),|\cdot|_H)\rightarrow  (H,|\cdot|_B)$ is a full-range
isometry.
This implies the following facts:
\begin{enumerate}[(1)]
\setlength\itemsep{0.05em}
\item 
there exists a unique extension $\widetilde
R\colon H\to H_B$;
\item $\widetilde R$ and $\widetilde R^{-1}$ are
isometries;
\item $\widetilde R^{-1}=\widetilde{R^{-1}}$, where
$\widetilde{R^{-1}}\colon H_B\to H$ is the unique continuous extension of
$R^{-1}$.
\end{enumerate}
Denote by $\overline R$ the operator $\widetilde R$
considered as an operator $H_B\supset H= D(\overline R)\rightarrow H_B$.  The above facts imply that $\overline R$ is a
densely-defined full-range closed linear operator in $H_B$, and that $D(R)$ is a core for $\overline R$.

\medskip
We will need the following proposition.

\begin{proposition}\label{2015-09-28:06}
  Let $R\colon D(R)\subset H\rightarrow H$ be a densely defined linear operator such that $R^{-1}\in L(H)$. Let $H_B$ be the Hilbert space defined above as the completion of $H$ with respect to the norm $|\cdot|_B$ given by \eqref{eq:2016-05-19:00}.
  \begin{enumerate}[(i)]
  \item\label{2016-05-19:01} Suppose that
\begin{equation}
  \label{eq:2015-09-28:04}
  S_tR\subset RS_t\qquad \forall  t\in \mathbb{R}^+.
\end{equation}
Then, for every $t\in \mathbb{R}^+$, there exists a unique continuous extension $\overline S_t$ of $S_t$ to $H_B$, the family $\overline S\coloneqq\{\overline S_t\}_{t\in\mathbb{R}^+}$ is a strongly continuous semigroup of contractions on $H_B$, and
\begin{gather}
\overline{S}_t \overline{R}\subset \overline{R}\overline{S}_t\qquad \forall t\in\mathbb{R}^+
 \label{eq:2016-05-19:06},  \\
\overline A =\overline R A\overline R^{-1},
\label{2016-05-19:07}
\end{gather}
where $\overline A$ is the infinitesimal generator of $\overline S$.

\item\label{2016-05-19:02} 
Suppose that
  \begin{gather}
  AR=RA,   \label{2016-05-19:03}\\ 
 D(A)\subset D(R).\label{2016-05-19:04}
\end{gather}
Then
\eqref{eq:2015-09-28:04} is satisfied
and
 the Yosida approximations $\{\overline A_n\}_{n\geq 1}$ of the infinitesimal generator $\overline A$ of $\overline S$ are given by the unique continuous extensions to $H_B$ of the Yosida approximations $\{A_n\}_{n\geq 1}$ of $A$, i.e.\ 
$$
\overline A_n=\overline{A_n}\qquad \forall n\geq 1.
$$
\end{enumerate}
\end{proposition}
\begin{proof}
\emph{(\ref{2016-05-19:01})}
Suppose that
\eqref{eq:2015-09-28:04} holds true. 
Observe that \eqref{eq:2015-09-28:04} implies
\begin{equation}
  \label{eq:2015-06-04:03}
AR\subset RA.  
\end{equation}
Since
$R^{-1}S_t=S_tR^{-1}$, 
we have
$$
| S_tx|_B=| R^{-1}S_tx|_H=| S_tR^{-1}x|_H\leq
| R^{-1}x|_H=
| x|_B\qquad \forall  t\in \mathbb{R}^+,\, x\in H.
$$
We can then extend each $S_t$ to an operator $\overline S_t\in
\lin(H_B)$ with the operator norm less than or equal to $1$.
By density of $H$ in $H_B$,
 it is clear that the family $\{\overline S_t\}_{t\in \mathbb{R}^+}$
is a semigroup of contractions.
Moreover, for $x\in H$,
$$
\lim_{t\to 0^+}| \overline S_tx-x|_B= \lim_{t\rightarrow 0^+}
| R^{-1}(S_tx-x)|_H=
\lim_{t\rightarrow 0^+}| S_tR^{-1}x-R^{-1}x|_H=0.
$$
The above observations imply that the family $\{\overline S_t\}_{t\in \mathbb{R}^+}$ is uniformly  bounded 
and strongly continuous on a dense subspace of $H_B$. 
Thus,
by 
\cite[Proposition 5.3]{Engel2000},
$\overline S$ is a strongly continuous semigroup on $H_B$.

\smallskip

We 
now prove
\eqref{eq:2016-05-19:06}.
Let $(x,\overline Rx)\in \Gamma(\overline R)$, where $\Gamma(\overline R)$ is the graph of $\overline R$. 
We noticed that $D(R)$ is a core for $\overline R$.
Then we can choose a sequence $\{(x_n,Rx_n)\}_{n\in \mathbb{N}}\in \Gamma(R)$ such that $(x_n,Rx_n)\rightarrow ( x,\overline Rx)$ in $H_B\times H_B$. 
Hence, using \eqref{eq:2015-09-28:04}, we can write
$$
\overline S_t\overline R x=\lim_{n\rightarrow +\infty} \overline S_tRx_n=
\lim_{n\rightarrow +\infty}  S_tRx_n=
\lim_{n\rightarrow +\infty}  RS_tx_n=
\lim_{n\rightarrow +\infty}  \overline R\,\overline S_tx_n,
$$
where all the limits are considered in $H_B$. This means that $\{\overline R\,\overline S_tx_n\}_{n\in \mathbb{N}}$ is convergent in $H_B$. We recall that $\overline R$ is closed in $H_B$ and we observe that $\overline S_tx_n\rightarrow \overline S_t x$ in $H_B$ by continuity. Thus we conclude that
$\overline R\,\overline S_tx_n\rightarrow \overline R\,\overline S_t x $ in $H_B$. This proves
\eqref{eq:2016-05-19:06}.

\smallskip
Now let $\overline A$ be the generator of  the semigroup $\{\overline S_t\colon H_B\rightarrow H_B\}_{t\in \mathbb{R}^+}$. Obviously
$\overline A$ is an extension of $A$, i.e.\ $\overline Ax=Ax$ for $x\in D(A)$. We will show that 
$D(\overline A)=\overline R(D(A))$. Using
\eqref{eq:2016-05-19:06}
we have for $x\in H_B$, 
\[
    \lim_{t\to 0^+} \frac{\overline S_t-I}{t}x=
    \lim_{t\to 0^+} \frac{\overline S_t-I}{t}\overline R\,\overline R^{-1}x
  = (\mathrm{by}\ \eqref{eq:2016-05-19:06}
  )=
\lim_{t\to 0^+} \overline R\frac{\overline S_t-I}{t}\overline R^{-1}x
  =\lim_{t\to 0^+} \overline R\frac{ S_t-I}{t}\overline R^{-1}x.
  \]
The last limit exists in $H_B$ if and only if the limit
\begin{equation*}
    \lim_{t\to 0^+} \frac{ S_t-I}{t}\overline R^{-1}x
\end{equation*}
exists in $H$. Therefore we conclude that 
\begin{equation}
  \label{eq:2015-06-04:04}
D(\overline A)=\overline R(D(A))\qquad \mathrm{and}\qquad\overline A x=\overline RA\overline R^{-1}x \quad \forall x\in D(\overline A),
\end{equation}
which can be written as \eqref{2016-05-19:07}.


\smallskip
\emph{(\ref{2016-05-19:02})}
Let $\{A_{n}\}_{n\geq 1}$ be the Yosida approximations of $A$. 
 We begin by showing that 
\begin{equation}
  \label{eq:2015-08-13:02}
  (n-A)^{-1}R\subset R(n-A)^{-1}\ \ \ \ \forall n\geq 1.
\end{equation}
By 
\eqref{2016-05-19:04},
it follows that
$$
D((n-A)^{-1}R)= D(R)\subset H=D(R(n-A)^{-1}).
$$
By 
\eqref{2016-05-19:04}, we have, for $x\in D(R)$,
\begin{equation}
  \label{eq:2015-08-13:04}
  A(n-A)^{-1}x=n(n-A)^{-1}x-x\subset D(A)+D(R)\subset D(R),
\end{equation}
hence $(n-A)^{-1}x\in D(RA)$.
Then, by using
\eqref{2016-05-19:03}, we can write, 
for $x\in D(R)$,
$$
(n-A)^{-1}Rx=
(n-A)^{-1}R(n-A)(n-A)^{-1}x=(n-A)^{-1}(n-A)R(n-A)^{-1}x=R(n-A)^{-1}x.
$$
This shows \eqref{eq:2015-08-13:02}.

\smallskip
We now claim that 
\begin{equation}
  \label{eq:2016-05-19:05}
  e^{tA_n}R\subset Re^{tA_n},
\end{equation}
where $e^{tA_n}$ is the semigroup generated by $A_n$.
By \eqref{eq:2015-08-13:02}, we have
\begin{equation*}
    A_nRx=n^2(n-A)^{-1}Rx-nRx=
n^2R(n-A)^{-1}x-nRx=
RA_nx \qquad
\forall x\in D(R),
\end{equation*}
that is
\begin{equation}
  \label{eq:2015-08-13:01}
  A_nR\subset RA_n.
\end{equation}
Let $x\in D(R)$.  By 
\eqref{eq:2015-08-13:04}
and 
\eqref{eq:2015-08-13:01},
\begin{equation}
  \label{eq:2015-08-13:05}
  A_n^kRx=RA_n^kx\qquad \forall k\in \mathbb{N}.
\end{equation}
For $t\in \mathbb{R}^+$, define
$$
y_m\coloneqq \sum_{k=0}^m\frac{t^k}{k!}A_n^kx.
$$
By \eqref{eq:2015-08-13:04}, $y_m\in D(R)$.
 Moreover, $\lim_{m\rightarrow +\infty}y_m=e^{tA_n}x$ and, by \eqref{eq:2015-08-13:05},
$$
\lim_{m\rightarrow +\infty}Ry_m=\lim_{m\rightarrow +\infty}
\sum_{k=0}^m\frac{t^k}{k!}A_n^kRx=e^{tA_n}Rx.
$$
Since $R$ is closed, it follows that $e^{tA_n}x\in D(R)$, and $Re^{tA_n}x=e^{tA_n}Rx$.
Since this holds for every $x\in D(R)$, we conclude $e^{tA_n}R\subset Re^{tA_n}$.

\smallskip
We can now prove that 
\eqref{eq:2015-09-28:04}
is satisfied.
Let $x\in D(R)$.
By \eqref{eq:2016-05-19:05},
$$
\lim_{n\rightarrow \infty}Re^{tA_n}x
=
\lim_{n\rightarrow \infty}e^{tA_n}Rx
=
S_tRx.
$$
Since $R$ is closed, we have $\lim_{n\rightarrow \infty}e^{tA_n}x=
S_tx\in  D(R)$ and $RS_tx=S_tRx$.
Then
\eqref{eq:2015-09-28:04}
is verified.

\smallskip
We can now conclude the proof.
By \eqref{eq:2016-05-19:05},
arguing as it was done for $S$, we obtain that every $S_n$ can be uniquely extended to the semigroup $e^{t\overline {A_n}}$ on $H_B$
generated by 
$\overline {A_n}$.
Similarly to \eqref{eq:2015-06-04:04},
we have
\begin{equation}
  \label{eq:2015-09-23:02}
D(\overline{A_n})=\overline R(D(A_n))\qquad \mathrm{and}\qquad\overline{A_n} x=\overline RA_n\overline R^{-1}x \quad \forall x\in D(\overline{A_n}).
\end{equation}
We observe that $\overline R ( D(A_n) ) =\overline R(H)=H_B$.
If $x\in H$, by 
\eqref{2016-05-19:04},
 \eqref{eq:2015-06-04:04}, 
\eqref{eq:2015-08-13:02},
 and 
\eqref{eq:2015-09-23:02}, we have 
\begin{equation*}
  \begin{split}
    \overline {A_n}x=\overline RA_n\overline R^{-1}x&=
\overline R nA(n-A)^{-1}\overline R^{-1}x=
 R nA(n-A)^{-1} R^{-1}x
\\
&=n( R A R^{-1})( R(n-A)^{-1} R^{-1})x
=n( R A R^{-1})(n-A)^{-1}x
=
n A(n- A)^{-1}x,
\end{split}
\end{equation*}
which can be written as
$$
\overline{A_n}x=n\overline A(n-\overline A)^{-1}x=\overline A_nx\qquad \forall x\in H,
$$
where $\overline A_n$ is the Yosida approximation of $\overline A$.
Finally, since both $\overline {A_n}$ and $\overline A_n$ are continuous on $H_B$, and since $H$ is dense in $H_B$, we obtain
$$
\overline{A_n}=\overline A_n,
$$
and then
 $e^{t\overline A_n}=e^{t\overline {A_n}}$, where $e^{t\overline A_n}$ is the semigroup generated by $\overline A_n$.
\end{proof}

\vskip10pt
In the remaining of this section we will assume that \eqref{2016-05-19:03} and
\eqref{2016-05-19:04}
 hold true.


\begin{assumption}\label{ipot:B-continuity-FG}
    The functions $ G$ and $\Sigma$ are Lipschitz with respect to the norm $|\cdot|_B$, with respect to the second variable and uniformly in the
  first one, that is there exists $M>0$ such that
  \begin{equation}\label{eq:2015-06-04:02}
    |  G(t,x)- G(t,y)|_{B}+|\Sigma(t,x)-\Sigma(t,y)|_{L(\mathbb{R}^m,H_B)}\leq M|
x-y|_B
\end{equation}
for all $t\in [0,T]$, $x,y\in H$.
Denote by $\overline G$ (resp.\ $\overline \Sigma$) the unique extension of $G$ (resp.\ $\Sigma$) to a function from $[0,T]\times H_B$ into $H_B$ (resp.\ from $[0,T]\times H_B$ into $L(\mathbb{R}^m,H_B)$).
\end{assumption}

\begin{remark}\label{rem:assumptions}
It is obvious that Assumptions \ref{ass:2015-06-04:01}
and
\ref{ipot:B-continuity-FG}
are satisfied if
\begin{equation}
\label{eq:contsigma0}
  |  G(t,x)- G(t,y)|_H+|\sigma(t,x)-\sigma(t,y)|_{L(\mathbb{R}^m)}\leq M|
x-y|_B.
\end{equation}
It is then easy to see that the functions
$G_0(t,x)=G(t,Rx)$ and $\sigma_0(t,x)=\sigma(t,Rx)$ defined on $[0,T]\times D(R)$ satisfy
\begin{equation}
\label{eq:contsigma1}
  |  G_0(t,x)- G_0(t,y)|_H+|\sigma_0(t,x)-\sigma_0(t,y)|_{L(\mathbb{R}^m)}\leq M|x-y|_H
\end{equation}
for $t\in [0,T]$ and $x,y\in D(R)$, and hence they
uniquely extend
 to functions defined on $[0,T]\times H$ satisfying
\eqref{eq:contsigma1} for all $t\in [0,T]$ and $x,y\in H$.
The converse is also true, i.e.\ \eqref{eq:contsigma1} implies \eqref{eq:contsigma0}.
 Thus \eqref{eq:contsigma0} is satisfied if and only if 
$G(t,x)=G_0(t,R^{-1}x)$, $\sigma(t,x)=\sigma_0(t,R^{-1}x)$, for $(t,x)\in [0,T]\times H$, for some $G_0$, $\sigma_0$ which satisfy
\eqref{eq:contsigma1} for all $t\in [0,T]$ and $x,y\in H$. 
We notice that for $\sigma$, \eqref{eq:contsigma0} is
also necessary for Assumptions \ref{ass:2015-06-04:01}
and \ref{ipot:B-continuity-FG}.

 For instance,
focusing on $\sigma$ (which corresponds to $\nu$ in the financial problem considered in Section~\ref{sec:motivationfinance}),  this condition is easily seen to be satisfied if
\[
\sigma(t,x)=f(t,\langle x,\bar y^1\rangle,\ldots,\langle x,\bar y^n\rangle)
\]
for some $f\colon[0,T]\times \mathbb{R}^n\to L(\mathbb{R}^m)$ Lipschitz continuous in the last $n$ variables (uniformly for $t\in [0,T]$) and
$\bar y^1,\ldots,\bar y^n\in D(R^*)$.
Indeed, in such a case we can write
\begin{equation}
  \label{eq:2016-05-19:08}
  \sigma(t,x)=f(t,\langle x,\bar y^1\rangle,\ldots,\langle x,\bar y^n\rangle)
=
f(t,\langle R^{-1}x,R^*\bar y^1\rangle,\ldots,\langle R^{-1}x,R^*\bar y^n\rangle)=\sigma_0(t,R^{-1}x),
\end{equation}
where $\sigma_0(t,x)=f(t,\langle x,R^*\bar y^1\rangle,\ldots,\langle x,R^*\bar y^n\rangle)$.
Since later in \eqref{eq:Blambda1} we take $R=A-I$,
in applications to our financial problem (Section~\ref{sec:motivationfinance}) this would mean that 
$$
\bar y^i=({\bar y}_0^i,{\bar y}_1^i)\in \mathbb{R}\times W^{1,2}(\mathbb{R}^-)\qquad  i=1,\ldots,n.
$$
 Thus a function of the form
 \begin{equation}
   \label{eq:2016-09-08:00}
   \sigma(t,x)=f\left(t, x_0^1 \bar y_0^1,\int_{-\infty}^0x_1^1(s)\bar y_1^1(s) ds,\ldots,x_0^n \bar y_0^n,\int_{-\infty}^0x_1^n(s)\bar y_1^n(s) ds\right),
 \end{equation}
 where $f\colon[0,T]\times \mathbb{R}^{2n}\to L(\mathbb{R}^m)$ is continuous in the $2n+1$ variables and Lipschitz continuous in the last $2n$
variables, uniformly for $t\in [0,T]$, satisfies Assumptions \ref{ass:2015-06-04:01}
and \ref{ipot:B-continuity-FG}. 

One can also give an equivalent condition which may be easier to check.
We can only require that $G(t,x)=G_0(t,Kx)$, $\sigma(t,x)=\sigma_0(t,Kx)$, for some $G_0,\sigma_0$
satisfying \eqref{eq:contsigma1} for all $t\in [0,T]$ and $x,y\in H$, and a bounded operator $K$ on $H$ such that
$|Kx|_H\leq C|R^{-1}x|_H$ for all $x\in H$. The last requirement (see e.g.\ \cite[p.\ 429, Proposition B.1]{DaPrato2014})
is equivalent to $K^*(H)\subset (R^{-1})^*(H)=D(R^*)$. 
In particular, if $K$ is the orthogonal projection onto a finite dimensional
subspace $H_0$ of $H$, then we need  $H_0\subset D(R^*)$.
By assuming without loss of generality that $\bar y^1,\ldots,\bar y^n$ in 
 \eqref{eq:2016-05-19:08}  are orthonormal, then
the previous example is readily reduced to the present if $K$ is
 the orthogonal projection onto 
$\operatorname{span}\{\bar y^1,\ldots,\bar y^n\}$.

Though functions like $\sigma$ in \eqref{eq:2016-09-08:00} are of a very special form (they are cylindrical  as functions of $x_1^1,\ldots,x_1^n$), 
it should be noticed that they are in general not smooth, since $f(t,\cdot)$ is only assumed to 
be Lipschitz continuous.
\end{remark}

Under Assumption
\ref{ipot:B-continuity-FG}
we can consider the following SDE on $H_B$ 
\begin{equation}\label{SDE-HB}
  \begin{dcases}
     d\overline X_s = \big(\overline A\overline X_s + \overline G\big(s,\overline X_s\big)\big) \de s+ \overline \Sigma\big(s,\overline X_s\big) \de W_s, \qquad s\in(t, T],\\
    \overline X_t =  x, 
  \end{dcases}
\end{equation}
where $x \in H_B$.
 By changing the reference Hilbert space from $H$ to $H_B$, we can apply Proposition \ref{2015-09-27:11} and say that
 SDE \eqref{SDE-HB} has a unique mild solution $\overline X^{t, x }$ in
$\mathcal H_\mathcal{P}^p(H_B)$, and $[0,T]\times H_B\to \mathcal{H}^p_\mathcal{P}(H_B),\ (t, x )\mapsto \overline X^{t, x }$,
is continuous and $|\cdot|_B$-Lipschitz with respect to $ x $, uniformly in $t$.

\begin{proposition}
  \label{prop:2015-06-08:05}
For any $p\geq 2$,  under Assumptions
\ref{ass:2015-06-04:01}
and \ref{ipot:B-continuity-FG},
there exists a unique mild solution
 $\overline  X ^{t, x }\in \mathcal{H}^p_\mathcal{P}(H_B)$ of SDE \eqref{SDE-HB}, and the map
\begin{equation}
  \label{eq:2015-06-19:02}
  [0,T]\times H_B\rightarrow \mathcal{H}_\mathcal{P}^p(H_{B}),\,(t, x )\mapsto \overline X^{t, x }
\end{equation}
is continuous in $(t, x )$, and Lipschitz in $x$, uniformly in $t$. If $x\in H$, $\overline X^{t,x}\in \mathcal{H}_\mathcal{P}^p(H)$ and $\overline X^{t,x}=X^{t,x}$, 
where
$X^{t,x}\in \mathcal{H}^p_\mathcal{P}(H)$ is the unique mild solution of SDE \eqref{eq:2013-02-19:ab}.
\end{proposition}
\begin{proof}
The first part follows from Proposition \ref{2015-09-27:11}. It remains to comment on the fact that $X^{t,x}=\overline X^{t,x}$ if $x\in H$.
The space $\mathcal H_\mathcal{P}^p(H)$ is continuously
embedded in $\mathcal H^p_\mathcal{P}(H_B)$. Thus, if $G$ and
$\Sigma$ satisfy Assumptions
\ref{ass:2015-06-04:01}
and \ref{ipot:B-continuity-FG}, and if the initial value $ x $ belongs to $H$,
the mild solution $X^{t, x}$ of
\eqref{eq:2013-02-19:ab} is also a mild solution
of \eqref{SDE-HB}, and then, by  uniqueness of mild solutions, 
$X^{t,\overline x}=\overline X^{t, x }$  in $\mathcal{H}^p_\mathcal{P}(H_B)$.
\end{proof}

In order to obtain an a-priori estimate giving the regularity in which we are interested, we will need to approximate mild solutions with other mild solutions of SDEs with smoother coefficients.

\begin{proposition}
  \label{ss:2015-06-20:00}
 Let $G$ and $\sigma$ satisfy Assumptions  \ref{ass:2015-06-04:01}
and
\ref{ipot:B-continuity-FG}. There exist sequences $\{G_n\}_{n\in \mathbb{N}}\subset  C_b^{0,1}([0,T]\times H,H)$, $\{\Sigma_n\}_{n\in \mathbb{N}}\subset C_b^{0,1}([0,T]\times H,L(\mathbb{R}^m,H))$,
with
$\Sigma_n(t,x)y=(\sigma_n(t,x)y,0_1)$ for some
 $\sigma_n\in C_b^{0,1}([0,T]\times H,L(\mathbb{R}^m))$,
satisfying:
\begin{enumerate}[(i)]
\item 
For every $n\in \mathbb{N}$, $G_n$ and $\Sigma_n$ have extensions $\overline G_n\in C^{0,1}_b([0,T]\times H_B,H_B)$ and 
$\overline \Sigma_n \in C^{0,1}_b([0,T]\times H_B,L(\mathbb{R}^m,H_B))$.
\item  
For all $(t,x),\,(t,y)\in [0,T]\times H_B$,
\begin{gather}
  \sup_{n\in \mathbb{N}}      | \overline G_n(t,x)- \overline G_n(t,y)|_B  \leq M|x-y|_B\label{2015-09-24:07}\\
  \sup_{n\in \mathbb{N}} |\overline \Sigma_n(t,x)-\overline \Sigma_n(t,y)|_{L(\mathbb{R}^m,H_B)}\leq M|x-y|_B.\label{2015-09-24:08}
\end{gather}
\item 
For every compact set $K\subset H_B$,
\begin{gather}
  \lim_{n\rightarrow \infty}
\sup_{(t,x)\in [0,T]\times K}
|\overline G(t,x)-\overline G_n(t,x)|_B
=0 \label{eq:2015-09-24:02}
\\
  \lim_{n\rightarrow \infty}
\sup_{(t,x)\in [0,T]\times K}
|\overline \Sigma(t,x)-\overline \Sigma_n(t,x)|_{L(\mathbb{R}^m,H_B)}
=0.\label{eq:2015-09-24:02b}
\end{gather}
\end{enumerate}
\end{proposition}

\begin{remark}\label{2015-09-30:00}
We remark that, due to the fact that the range of 
$\Sigma$ is finite-dimensional (see \eqref{eq:2015-06-20:01}), once the above continuity/differentiability/approximation conditions for
 $\overline\Sigma_n$ are satisfied with respect to $H_B$, they automatically hold for $\Sigma_n$ with respect to $H$.
\end{remark}

\begin{proof}[Proof of Proposition~\ref{ss:2015-06-20:00}]
 The proof uses approximations similar to those in 
\cite{peszat1995}.
 Let $\{e_n\}_{n\in \mathbb{N}}$ be an orthonormal basis of $H_{B}$ contained in $H$. 
For $n\in\mathbb{N}$, let us define the functions
$$
I_n\colon \mathbb{R}^n\rightarrow H_B,\ y \mapsto \sum_{k=1}^{n}y_k e_k
$$
and
$$
 P_n\colon H_B\rightarrow \mathbb{R}^n,\ x \mapsto (\langle x,e_1\rangle_B,\ldots,\langle x,e_{n}\rangle_B).
$$
It is clear that $|I_n|_{L(\mathbb{R}^n,H_B)}=1$ and $|I_nP_n|_{L(H_B)}=1$.
We observe also that, for every $n\in \mathbb{N}$, the linear operator
$$
H_B\rightarrow H,\ x \mapsto I_nP_nx=\sum_{k=1}^{n}\langle x,e_k\rangle_Be_k
$$
is well defined and continuous.
Denote $c_n\coloneqq |I_nP_n|_{L(H_B,H)}$.

Let
$$
\varphi(r)\coloneqq \begin{dcases}
  e^{-\frac{1}{1-r^2}}&\mathrm{if}\ r\in(-1,1)\\
  0&\mathrm{otherwise,}
\end{dcases}
$$
and, for every $n\in \mathbb{N}$,
$$
C_n\coloneqq  \left( \int_{\mathbb{R}^n}\varphi(n|y|_n)dy \right) ^{-1},
$$
where $|\cdot|_n$ denotes the Euclidean norm in $\mathbb{R}^n$.
Define
$$
g_n\colon [0,T]\times \mathbb{R}^n\rightarrow H
$$
by standard mollification
$$
g_n(t,y)\coloneqq C_n \left( G(t,I_n\cdot)*\varphi(n|\cdot|_n) \right) (y)=C_n\int_{\mathbb{R}^n}G \left( t,\sum_{k=0}^{n-1}z_k e_k \right) \varphi(n|y-z|_n)dz,
$$
for all $(t,y)\in [0,T]\times \mathbb{R}^n$.
We observe that $g_n$ is well-defined, because $G$ is $H$-valued and continuous, and $\varphi$ has compact support.
By Lebesgue's dominated convergence theorem, 
$g_n$ is continuous.

Since the map $\mathbb{R}^n \rightarrow \mathbb{R},\ z \mapsto \varphi(n|z|)$,  is continuously differentiable and has compact support and since $G$ is continuous, by a standard argument we can differentiate under the integral sign to obtain
$g_n$
 is differentiable with respect to $y$ and
$$
D_yg_n(t,y)v=
nC_n\int_{\mathbb{R}^n}G \left( t,I_nz
\right) \varphi'(n|y-z|_n)\frac{\langle y-z,v \rangle_n}{|y-z|_n}dz.
$$
By Lebesgue's dominated convergence theorem, the map
$$
[0,T]\times \mathbb{R}^n\times \mathbb{R}^n\rightarrow H,\ 
(t,y,v) \mapsto 
D_yg_n(t,y)v
$$ 
is continuous. 
Thus $g_n\in C^{0,1} \left( [0,T]\times \mathbb{R}^n,H \right) $.
Define
 $$
 \overline G_n\colon [0,T]\times H_B \rightarrow H_B
$$
by
$$
\overline G_n(t,x)\coloneqq
g_n(t,P_nx)=C_n\int_{\mathbb{R}^n}G \left( t,I_nP_nx-I_nz
\right) \varphi(n|z|_n)dz\ \ \ \ \forall (t,x)\in [0,T]\times H_B.
$$
Since $\overline G_n([0,T]\times H_B)\subset H$,  we can also define $G_n\colon [0,T]\times H\rightarrow H$ by $G_n(t,x)\coloneqq \overline G_n(t,x)$ for every $(t,x)\in [0,T]\times H$. 
Then
$G_n\in C^{0,1}([0,T]\times H,H)$ and $\overline G_n\in C^{0,1}([0,T]\times H_B,H_B)$.
Moreover, by Assumption  \ref{ass:2015-06-04:01}, 
\begin{equation}\label{2015-09-24:00}
  \begin{split}
      |G_n(t,x)-&G_n(t,x')|_H=
      |g_n(t,P_nx)-g_n(t,P_nx')|_H \\
&\leq C_n\int_{\mathbb{R}^n}\left|G \left( t,I_nP_nx-I_nz
\right)
-G \left( t,I_nP_nx'-I_nz
\right)\right|_H
 \varphi(n|z|_n)dz\\
&\leq M\left|I_nP_nx-I_nP_nx'\right|_H\leq Mc_n|x-x'|_B\leq Mc_n|R^{-1}|_{L(H)}|x-x'|_H,
\end{split}
\end{equation}
for every $t\in[0,T]$ and $x,x'\in H$.
Similarly, by Assumption \ref{ipot:B-continuity-FG},
\begin{equation}\label{2015-09-24:01}
  \begin{split}
          |\overline G_n(t,x)-\overline G_n(t,x')|_B&=      |g_n(t,P_nx)-g_n(t,P_nx')|_B\\
& \leq M\left|I_nP_nx-I_nP_nx'\right|_B
\leq M\left|x-x'\right|_B,
\end{split}
\end{equation}
for every $t\in[0,T]$ and $x,x'\in H_B$.
Thus $G_n\in C_b^{0,1}([0,T]\times H,H)$ and $\overline G_n\in C_b^{0,1}([0,T]\times H_B,H_B)$.

To prove \eqref{eq:2015-09-24:02}
for every compact $K\subset H_B$, we first notice that
\[
\sup_{x\in K}|I_nP_nx-x|_B=\epsilon_n\to 0\quad\text{as}\,\,n\to+\infty.
\]
Thus
by \eqref{eq:2015-06-04:02},
\begin{equation}\label{2015-09-24:06}
\lim_{n\rightarrow +\infty}   \sup_{(t,x)\in [0,T]\times K}| \overline G(t,I_nP_nx)-\overline G(t,x)|_B
\leq \lim_{n\rightarrow +\infty} M\epsilon_n=0.
\end{equation}
Moreover, for $(t,x)\in[0,T]\times H_B$,
\begin{equation*}
  \begin{split}
   & |\overline G(t,I_nP_nx)-\overline G_n(t,x)|_{B}\leq
C_n\int_{\mathbb{R}^n}\left|G \left( t,I_nP_nx- I_nz
\right)
-G \left( t,I_nP_nx\right)
\right|_B \varphi(n|z|_n)dz\\
&\quad\quad
\leq
MC_n\int_{\mathbb{R}^n}|I_nz
|_B \varphi(n|z|_n)dz\leq
MC_n\int_{\mathbb{R}^n}|z
|_n \varphi(n|z|_n)dz\leq \frac{M}{n}.
  \end{split}
\end{equation*}
This, together with \eqref{2015-09-24:06}, gives \eqref{eq:2015-09-24:02}.

We have thus proved that $\{G_n\}_{n\in \mathbb{N}}\subset C^{0,1}_b([0,T]\times H,H)$, 
that $\{\overline G_n\}_{n\in \mathbb{N}}\subset C^{0,1}_b([0,T]\times H_B,H_B)$,  and 
that \eqref{2015-09-24:07}
and \eqref{eq:2015-09-24:02} hold true.

The other half of the proof, regarding $\Sigma$, is similar. We only make a few comments.
For $n\in \mathbb{N}$, define
$$
\zeta_n\colon \mathbb{R}^n\rightarrow L(\mathbb{R}^m)
$$
by
$$
\zeta_n(t,y)\coloneqq C_n \big( \sigma(t,I_n\cdot)*\varphi(n|\cdot|_n) \big) (y)
=C_n\int_{\mathbb{R}^n}\sigma \left( t,\sum_{k=1}^{n}z_k e_k \right) \varphi(n|y-z|_n)dz,
$$
for all $(t,y)\in [0,T]\times \mathbb{R}^n$, and
$
\overline \sigma_n\colon [0,T]\times H_B\rightarrow L(\mathbb{R}^m)
$
by  $\overline \sigma_n(t,x)\coloneqq \zeta_n(t,I_nP_nx)$ for all 
$(t,y)\in [0,T]\times \mathbb{R}^n$, and $n\in \mathbb{\mathbb{N}}$.
 Arguing as it was done done for $g_n$, we have that $\zeta_n\in C^{0,1}([0,T]\times \mathbb{R}^n,L(\mathbb{R}^m))$, 
and then $\overline \sigma_n \in C^{0,1}([0,T]\times H_B,L(\mathbb{R}^m))$. Moreover,
 \begin{equation}\label{2015-09-24:09}
   |\overline \sigma_n(t,x)-\overline \sigma_n(t,x')|_B     
   \leq M\left|x-x'\right|_B,
 \end{equation}
and hence $\overline \sigma_n\in C_b^{0,1}([0,T]\times H_B,L(\mathbb{R}^m))$. 
The proof of \eqref{eq:2015-09-24:02b} is done in the same way as for $\overline G_n$.
Finally we define
$$
\overline
\Sigma_n(t,x)y\coloneqq
(\overline 
\sigma_n(t,x)y,0_1)
\ \ \ \forall (t,x)\in[0,T]\times H_B, \ \forall y\in \mathbb{R}^n, \ \forall n\in \mathbb{N}
$$
and
$$
\Sigma_n(t,x)y\coloneqq 
\overline \Sigma_n(t,x)y
\ \ \ \forall (t,x)\in[0,T]\times H, \ \forall y\in \mathbb{R}^n, \ \forall n\in \mathbb{N}.
$$
This concludes the proof.
\end{proof}

\vskip10pt
Unless specified otherwise, Assumptions  \ref{ass:2015-06-04:01} and \ref{ipot:B-continuity-FG} will be standing for 
the remaining part of the
manuscript,
and
$\{ G_n\}_{n\in \mathbb{N}}$,
$\{ \Sigma_n\}_{n\in \mathbb{N}}$,
 $\{\overline G_n\}_{n\in \mathbb{N}}$,
$\{\overline \Sigma_n\}_{n \mathbb{N}}$ will denote the sequences introduced in Proposition
\ref{ss:2015-06-20:00}.

\vskip10pt

Let $\{A_n\}_{n\geq 1}$ be the Yosida approximation of $A$. We recall that
for every $n\geq 1$, by Proposition \ref{2015-09-28:06}, $A_n$ has a unique  continuous
extension $\overline {A_n}$ to $H_B$, and $\overline{A_n}=\overline A_n$, where $\{\overline A_n\}_{n\geq 1}$ 
is the Yosida approximation of the infinitesimal generator $\overline A$ of
$\overline S$. We remind that we denote by $e^{t\overline {A}_n}$ the semigroup generated by $\overline {A}_n$. 
For $t\in [0,T)$
and $n\geq 1$,   
we denote by $X_{n}^{t, x }$, $\overline X_{n}^{t, x }$ 
respectively the unique mild solutions to
\begin{equation}\label{eq:2015-06-23:08}
  \begin{dcases}
     d X_{n,s} = \left(  A_n  X_{n,s} +   G_n\left(s,  X_{n,s}\right)\right) \de s+   \Sigma_n\left(s,  X_{n,s}\right) \de W_s \qquad s\in(t, T]\\
      X_{n,t} =  x\in H,
  \end{dcases}
\end{equation}
\begin{equation}\label{eq:2015-06-21:00}
  \begin{dcases}
     d\overline X_{n,s} = \big(\overline A_n\overline X_{n,s} + \overline G_n\big(s,\overline X_{n,s}\big)\big) \de s+ \overline \Sigma_n
     \big(s,\overline X_{n,s}\big) \de W_s \qquad s\in (t, T]\\
    \overline X_{n,t} =  x\in H_B.
  \end{dcases}
\end{equation}
For any $p\geq 2$, existence and uniqueness of mild solution are provided by Propositions \ref{2015-09-27:11} and
\ref{prop:2015-06-08:05}, together with the continuity of the maps
  \begin{equation}
    [0,T]\times H\rightarrow \mathcal{H}_\mathcal{P}^p(H),\,(t,x)\mapsto X_n^{t,x}\qquad
    [0,T]\times H_B\rightarrow \mathcal{H}_\mathcal{P}^p(H_B),\,(t,x)\mapsto \overline X_n^{t,x}.
  \end{equation}



\begin{proposition}
  \label{prop:2015-06-20:03}
Let Assumptions
\ref{ass:2015-06-04:01} and \ref{ipot:B-continuity-FG} hold and let
$p\geq 2$.
Then:
\begin{enumerate}[(i)]
\item\label{2015-09-28:07} For every $n\in \mathbb{N}$ and $x\in H$, $X^{t,x}_n=\overline X^{t,x}_n$ (in $\mathcal{H}^p_\mathcal{P}(H_B)$).
\item \label{2015-09-28:08}
$
\lim_{n\rightarrow +\infty}\overline X^{t, x }_n= \overline X^{t, x }$ in $\mathcal{H}_\mathcal{P}^p(H_B)$ 
uniformly for $(t, x )$ on compact sets of $[0,T]\times H_B$.
\item \label{2015-09-28:09}
For every $n\in \mathbb{N}$ the map
\begin{equation}
  \label{eq:2015-09-29:00}
  [0,T]\times H_B\rightarrow \mathcal{H}_\mathcal{P}^p(H_B),\, (t, x )\mapsto \overline X^{t, x }_n
\end{equation}
belongs to $\mathcal{G}^{0,1}_s([0,T]\times H_B,\mathcal{H}^p_\mathcal{P}(H_B))$.
\item \label{2015-09-28:09b} The set $\{\nabla_{ x } \overline X_{n}^{t, x }\}_{n\in \mathbb{N}}$ is
bounded
 in $L(H_B,\mathcal{H}_\mathcal{P}^p(H_B))$,
uniformly for $(t,x)\in [0,T]\times H_B$.
\end{enumerate}
\end{proposition}
\begin{proof}
\emph{(\ref{2015-09-28:07})} 
Let $(t,x)\in[0,T]\times H$.
Since $A_n=\overline A_n$ on $H$, we have $e^{sA_n}=e^{s\overline A_n}$ on $H$ for all $s\in\mathbb{R}^+$. Recalling that 
$\mathcal{H}^p_\mathcal{P}(H)$ is continuously embedded in $\mathcal{H}^p_\mathcal{P}(H_B)$, 
we then have that the mild solution $X^{t,x}_n$ is also a mild solution of \eqref{eq:2015-06-21:00} in $\mathcal{H}^p_\mathcal{P}(H_B)$. By uniqueness we conclude that $X_n^{t,x}=\overline X^{t,x}_n$ in $\mathcal{H}^p_\mathcal{P}(H_B)$.

\emph{(\ref{2015-09-28:08})} For $t\in [0,T]$, $x\in H_B$, $n\geq 1$ , similarly to what was done in the proof of Proposition \ref{2015-09-27:11}, we  define the maps
$$
\overline \Phi(t;\cdot,\cdot)\colon H_B\times \mathcal{H}^p_\mathcal{P}(H_B)\rightarrow \mathcal{H}^p_\mathcal{P}(H_B)
$$
by
$$
\overline \Phi(t;x,Z)_s\coloneqq 
  \begin{dcases}
  x&s\in[0,t)\\
  \overline S_{s-t}x+\int_t^s \overline S_{s-w}
\overline G(w,Z_w) dw +\int_t^s \overline S_{s-w}\overline \Sigma(w,Z_w) dW_w&s\in[t,T]
\end{dcases}
$$
and
$$
\overline \Phi_n(t;\cdot,\cdot)\colon H_B\times \mathcal{H}^p_\mathcal{P}(H_B)\rightarrow \mathcal{H}^p_\mathcal{P}(H_B)
$$
by
$$
\overline \Phi_n(t;x,Z)_s\coloneqq 
  \begin{dcases}
  x&s\in[0,t)\\
  e^{(s-t)\overline A_n}x+\int_t^s e^{(s-w)\overline A_n}
\overline G_n(w,Z_w) dw +\int_t^s e^{(s-w)\overline A_n}\overline \Sigma_n(w,Z_w) dW_w&s\in[t,T].
\end{dcases}
$$
The mild solutions $\overline X^{t,x}$ and $\overline X^{t,x}_n$ are the fixed points  of $\overline \Phi(t;x,\cdot) $ and $\overline \Phi_n(t;x,\cdot)$  respectively.
 Since the operators $\overline A_n$, $n\geq 1$, are the Yosida approximations of $\overline A$, they generate semigroups of contractions
on $H_B$.
Recalling \eqref{2015-09-24:07} and \eqref{2015-09-24:08}, and arguing for $\overline \Phi$ and $\overline \Phi_n$ as in the proof of 
Proposition \ref{2015-09-27:11}
for $\Phi$, we find $\gamma>0$, depending only on 
$p$, $T$, $M$,
 such that 
 \begin{equation}\label{2015-09-28:16}
   \sup_{(t,x)\in[0,T]\times H}|\overline \Phi(t;x,Z)-\overline \Phi(t;x,Z')|_{\mathcal{H}^p_\mathcal{P}(H_B),\gamma}\leq \frac{1}{2}
|Z-Z'|_{\mathcal{H}^p_\mathcal{P}(H_B),\gamma} \ \ \ \forall Z,Z'\in \mathcal{H}^p_\mathcal{P}(H_B)
\end{equation}
and
 \begin{equation}\label{2015-09-28:17}
   \sup_{\substack{(t,x)\in[0,T]\times H\\n\in \mathbb{N}}}|\overline \Phi_n(t;x,Z)-\overline \Phi_n(t;x,Z')|_{\mathcal{H}^p_\mathcal{P}(H_B),\gamma}\leq \frac{1}{2}
|Z-Z'|_{\mathcal{H}^p_\mathcal{P}(H_B),\gamma} \ \ \ \forall Z,Z'\in \mathcal{H}^p_\mathcal{P}(H_B).
\end{equation}
Let $\{t_n\}_{n\geq 1}$ be a sequence converging to $t$ in $[0,T]$.
We claim that
\begin{equation}
  \label{eq:2015-09-28:11}
  \lim_{n\rightarrow +\infty}\overline \Phi_n(t_n;x,Z)=\overline \Phi(t;x,Z)
\  \mathrm{in}\ \mathcal{H}^p_\mathcal{P}(H_B),
\ \forall (x,Z)\in H_B\times \mathcal{H}^p_\mathcal{P}(H_B).
\end{equation}
Once \eqref{eq:2015-09-28:11} is proved, we can conclude, again invoking \cite[Theorem 7.1.5]{DaPrato2004}, that 
\begin{equation}
  \label{eq:2015-09-28:20}
  \lim_{n\rightarrow +\infty}\overline X^{t_n,x}_n=\overline X^{t,x}
 \  \mathrm{in}\ \mathcal{H}^p_\mathcal{P}(H_B),\ \forall x\in H_B.
\end{equation}
But \eqref{eq:2015-09-28:11} 
is easily obtained by combining strong convergence of $e^{t\overline A_n}$ to $\overline S_t$ uniformly for $t\in[0,T]$, sublinear growth of $\overline G_n(t,x)$ and $\overline \Sigma_n(t,x)$ in $x$ uniformly on $t\in[0,T]$ and $n\geq 1$ (obtained by \eqref{2015-09-24:07}, \eqref{2015-09-24:08}, \eqref{eq:2015-09-24:02}, \eqref{eq:2015-09-24:02b}, and by continuity of $\overline G(\cdot,0)$ and of $\overline \Sigma(\cdot,0)$), Burkholder-Davis-Gundy's inequality, Lebesgue's dominated convergence theorem, and pointwise convergence of $\{\overline G_n\}_{n\in \mathbb{N}}$ to $\overline G$ and of $\{\overline \Sigma_n\}_{n\in \mathbb{N}}$ to $\overline \Sigma$.

 By linearity of $\overline \Phi(t;x,Z)$ and $\overline \Phi_n(t;x,Z)$ in $x$, we have
 \begin{equation}
   \label{2015-09-28:12}
   \sup_{(t,Z)\in[0,T]\times \mathcal{H}^p_\mathcal{P}(H_B)}|\overline \Phi(t;x,Z)-\overline \Phi(t;x',Z)|_{\mathcal{H}^p_\mathcal{P}(H_B),\gamma}\leq
|x-x'|_{H_B} \ \ \ \forall x,x'\in H_B.
\end{equation}
and
 \begin{equation}
   \label{2015-09-28:15}
   \sup_{\substack{(t,Z)\in[0,T]\times \mathcal{H}^p_\mathcal{P}(H_B)\\n\in \mathbb{N}}}|\overline \Phi_n(t;x,Z)-\overline \Phi_n(t;x',Z)|_{\mathcal{H}^p_\mathcal{P}(H_B),\gamma}\leq
|x-x'|_{H_B} \ \ \ \forall x,x'\in H_B.
\end{equation}
Thus, using
\eqref{2015-09-28:16},
\eqref{2015-09-28:17}, and \cite[inequality ($***$) on p.\  13]{Granas2003},  
we have
 \begin{equation}\label{2015-09-28:18}
   \sup_{t\in[0,T]}|\overline X^{t,x}-\overline X^{t,x'}
|_{\mathcal{H}^p_\mathcal{P}(H_B),\gamma}\leq
2|x-x'|_{H_B} \ \ \ \forall x,x'\in H_B.
\end{equation}
and
 \begin{equation}\label{2015-09-28:19}
   \sup_{\substack{t\in[0,T]\\n\in \mathbb{N}}}|\overline X_n^{t,x}-\overline X_n^{t,x'}
|_{\mathcal{H}^p_\mathcal{P}(H_B),\gamma}\leq
2|x-x'|_{H_B} \ \ \ \forall x,x'\in H_B.
\end{equation}
Now \eqref{eq:2015-09-28:20}, \eqref{2015-09-28:18},  \eqref{2015-09-28:19}, and the continuity of $\overline{X}^{t,x}$ in $t$ 
(Proposition \ref{prop:2015-06-08:05}),
 yield the convergence of $\overline X^{t,x}_n$ to $\overline X^{t,x}$ in $\mathcal{H}^p_\mathcal{P}(H_B)$ as $n\rightarrow +\infty$, uniformly for $(t,x)$ on compact subsets of $[0,T]\times H_B$.

\emph{(\ref{2015-09-28:09})}
Let $n\geq 1$.
 By \cite[p.\ 243, Th.\ 9.8]{DaPrato2014}, for every $t\in[0,T]$, the map
\eqref{eq:2015-09-29:00}
 is G\^ateaux differentiable and, for every $x,y\in H_B$, the directional derivate $\nabla_x \overline X_n^{t,x}y$ is the  unique fixed point of  $\Psi_n(t,x;y,\cdot)$, where $\Psi_n(t,\cdot;\cdot)$ is defined by
$$
\Psi_n(t,x;\cdot,\cdot)\colon H_B\times \mathcal{H}^p_\mathcal{P}(H_B)
\rightarrow
\mathcal{H}^p_\mathcal{P}(H_B),
$$
$$
\hskip-10pt
\Psi_n(t,x;y,Z)\coloneqq
\begin{dcases}
  y&s\in[0,t)\\
  \overline S_{s-t}y+\int_t^s\overline S_{s-w}
  \nabla_x \overline G_n(w,\overline X_ {n,w}^{t,x})Z_w dw +\int_t^s \overline S_{s-w}\nabla_x\overline \Sigma_n(w,\overline X_ {n,w}^{t,x})Z_w dW_w&s\in[t,T].
\end{dcases}
$$
To show strong continuity and uniform boundedness of $\nabla_x \overline X_n$, we argue similarly as in the proof of Proposition
\ref{2015-09-27:11}.
By
\eqref{2015-09-24:07} and \eqref{2015-09-24:08},
\begin{gather}
      |\nabla_x \overline G_n(t,x)y|_B\leq M|y|_B\\
|\nabla_x \overline \Sigma_n(t,x)y|_{L(\mathbb{R}^m,H_B)}\leq M|y|_B,
\end{gather}
for all $(t,x,y)\in [0,T]\times H_B\times H_B$.
Then, by linearity of $\Psi_n(t,x;\cdot,\cdot)$ and by 
 \cite[Ch.\ 7, Proposition 7.3.1]{DaPrato2004},
there exists $\gamma>0$, depending only on 
$p$, $T$, $M$, $b$, 
 such that 
 \begin{multline}\label{2015-09-28:01}
        \sup_{(t,x)\in[0,T]\times H}|\Psi_n(t,x;y,Z)-\Psi_n (t,x;y,Z')|_{\mathcal{H}^p_\mathcal{P}(H),\gamma}=
   \sup_{(t,x)\in[0,T]\times H}|\Psi_n (t,x;0,Z-Z')|_{\mathcal{H}^p_\mathcal{P}(H),\gamma}\\
\leq \frac{1}{2}
|Z-Z'|_{\mathcal{H}^p_\mathcal{P}(H),\gamma} \ \ \ \forall Z,Z'\in \mathcal{H}^p_\mathcal{P}(H).
\end{multline}
We also have 
\begin{equation}
  \label{2015-09-28:00}
  \sup_{(t,x,Z)\in [0,T]\times H_B\times \mathcal{H}_\mathcal{P}^p(H_B)}
|\Psi_n (t,x;y,Z)-\Psi_n (t,x;y',Z)|_{\mathcal{H}^p_\mathcal{P}(H_B),\gamma}\leq  |y-y'|_{H_B}
\end{equation}
for all $y,y'\in H_B$.
By \eqref{2015-09-28:01}, \eqref{2015-09-28:00}, and 
\cite[inequality ($***$) on p.\ 13]{Granas2003},
 we thus obtain
 \begin{equation}
   \label{eq:2015-09-28:02} 
   \sup_{(t,x)\in [0,T]\times H_B}|\nabla_x \overline X_n^{t,x}y|_{\mathcal{H}^p_\mathcal{P}(H_B),\gamma}\leq 2|y|_{H_B}\ \ \ \ \forall y\in H_B.
 \end{equation}
Hence $\nabla_x\overline X_n^{t,x}$ is bounded in $L \left( H_B,\mathcal{H}^p_\mathcal{P}(H_B) \right) $, uniformly for $(t,x)\in[0,T]\times H_B$ and
$n\geq 1$.

Let 
  now $\{(t_k,x_k)\}_{k\in \mathbb{N}}\subset [0,T]\times H_B$ be a sequence  converging to $(t,x)\in[0,T]\times H_B$.
We claim that
\begin{equation}
  \label{2015-09-28:03}
  \lim_{k\rightarrow +\infty}\Psi_n (t_k,x_k;y,Z)=\Psi_n (t,x;y,Z)\ \ \ \ \forall (y,Z)\in H_B\times \mathcal{H}^p_\mathcal{P}(H_B).
\end{equation}
Once \eqref{2015-09-28:03} is proved, using 
\eqref{2015-09-28:01} and applying \cite[Theorem 7.1.5]{DaPrato2004}, we obtain
$$
\lim_{k\rightarrow +\infty}\nabla_x \overline X_n^{t_k,x_k}y=
\nabla_x \overline X_n^{t,x}y\ \ \ \ \forall y\in H_B,
$$
which provides the strong continuity of $\nabla_x\overline X_n^{t,x}$.
Recalling that $\lim_{k\rightarrow +\infty}\overline X_n^{t_k,x_k}=\overline X_n^{t,x}$ in $\mathcal{H}^p_\mathcal{P}(H_B)$, 
we can consider
 a subsequence, again denoted by $\{(t_k,x_k)\}_{k\in \mathbb{N}}$, 
such that $\lim _{n\rightarrow +\infty}\overline X_n^{t_k,x_k}=\overline X_n^{t,x}$ $ \mathbb{P} \otimes  dt$-a.e.\ 
on $\Omega_T$. Then \eqref{2015-09-28:03} is obtained by 
applying Lebesgue's dominated convergence theorem, together with Burkholder-Davis-Gundy's inequality, for the stochastic integral.

\emph{(\ref{2015-09-28:09b})} This follows immediately from \eqref{eq:2015-09-28:02}.
\end{proof}


\vskip10pt
We will make a particular choice of $R$ and thus $B$.
Recall that
$(0,+\infty)$ is contained in the resolvent set of $A$ (and hence of $A^*$). 
For  $\lambda>0$, let $A_\lambda\coloneqq A-\lambda$, $A^*_\lambda\coloneqq A^*-\lambda=(A-\lambda)^*$. If $R=A_\lambda$, then 
\eqref{eq:2015-09-28:04},
\eqref{2016-05-19:03}, and
\eqref{2016-05-19:04},
 are satisfied.
 We can then apply all of the above arguments with
\[
  B=B_{A,\lambda}\coloneqq   (A_\lambda^*)^{-1}A_\lambda^{-1}.
\]
Notice that
$$
|x|_{B_{A,\lambda}}\leq \left(1+|\lambda-\lambda'||A_\lambda^{-1}|_{L(H)}\right)|x|_{B_{A,\lambda'}}\qquad \forall  \lambda,\lambda'\in (0,+\infty),\, x\in H,
$$
hence the norms
$|\cdot|_{B_{A,\lambda}}$ and $|\cdot|_{B_{A,\lambda'}}$ are
equivalent.  We will thus pick $\lambda=1$ and from now on we set
\begin{equation}
  \label{eq:Blambda1}
  B\coloneqq B_{A,1}= (A_1^*)^{-1}A_1^{-1}.
\end{equation}
We observe that with this choice of $B$ we have
\[
 |x|_B=|(\overline A-I)^{-1}x|_H \quad\text{for all}\,\,x\in H_B,
\]
and
\[
 \langle x,y\rangle_B=\langle(\overline A-I)^{-1}x,(\overline A-I)^{-1}y\rangle \quad\text{for all}\,\,x,y\in H_B.
\]
In particular
\[
 \langle x,y\rangle_B=\langle (A^*-I)^{-1}(\overline A-I)^{-1}x,y\rangle \quad\text{if}\,\,x\in H_B,y\in H.
\]

\section{Viscosity solutions of Kolmogorov PDEs in Hilbert spaces with finite-dimensional second-order term}
\label{Sec:visc}

We remind that throughout the rest of the paper $B$ is defined by  \eqref{eq:Blambda1}. For this $B$, Assumptions~\ref{ass:2015-06-04:01} 
and \ref{ipot:B-continuity-FG} will be standing for the remaining part of the manuscript, 
$\{ G_n\}_{n\in \mathbb{N}}$,
$\{ \Sigma_n\}_{n\in \mathbb{N}}$,
 $\{\overline G_n\}_{n\in \mathbb{N}}$,
$\{\overline \Sigma_n\}_{n \mathbb{N}}$ denote the sequences introduced in Proposition
\ref{ss:2015-06-20:00}, the operators $A_n,n\geq 1$ are the Yosida approximations of $A$, and  $X_n^{t,x}$, $\overline X^{t,x}_n$ are respectively the mild solutions of \eqref{eq:2015-06-23:08},
\eqref{eq:2015-06-21:00}, with $B=B_{A,1}$, $n\geq 1$.
We recall that, by Proposition \ref{prop:2015-06-20:03},
$X^{t,x}=\overline X^{t,x}$ and $X_n^{t,x}=\overline X_n^{t,x}$ for every $(t,x)\in [0,T]\times H$, $n\geq 1$.

\subsection{Existence and uniqueness of solution}
The following assumption
will be standing for the remaining part of the work.
\begin{assumption}
  \label{ass:2015-06-21:03}
The function $h\colon H_{B}\rightarrow \mathbb{R}$ is such that there is a constant $M\ge 0$ such that
\begin{equation}\label{eq:hcontB}
 |h(x)-h(y)|\leq M|x-y|_B\quad
\forall x,y\in H.
\end{equation}
 \end{assumption}
The function $h$ extends uniquely to $\overline h:H_B\rightarrow \mathbb{R}$ which also satisfies \eqref{eq:hcontB}.
Taking the inf-sup convolutions of $\overline h$ in $H_B$ (see \cite{FGS,LL}) we can obtain a sequence of functions 
$\{\overline h_n\}_{n\in \mathbb{N}}\subset  C^1_b(H_{B})$ 
(and even more regular) such that
\begin{equation}\label{eq:supinfconvappr}
\sup_{\substack{n\in \mathbb{N}\\x\in H_{B}}}|D\overline h_n(x)|_{B}<+\infty
\qquad\mbox{and}\qquad
\lim_{n\rightarrow +\infty}\sup_{x\in H_B}\left|\overline h(x)-\overline h_n(x)\right|=0.
\end{equation}
The restriction of $\overline h_n$ to $H$ will be denoted by $h_n$.

We define the functions
\begin{align}
&  u\colon [0,T]\times H\to \mathbb{R},\ (t,x)\mapsto \mathbb{E}
  \left[h(X^{t,x}_{T})\right],\label{eq:2015-01-16:00}\\
 & u_n\colon [0,T]\times H\to \mathbb{R},\ (t,x)\mapsto \mathbb{E}
  \left[h_n(X^{t,x}_{n,T})\right],\  \ \ \ n\geq 1.\label{eq:2015-01-16:00a}
\end{align}
By sublinear growth of $h$ and $h_n$, $u$ and $u_n$ are well defined.
Each of the above functions has an associated Kolmogorov equation in $(0,T]\times H$. However we will only need to consider
the equation satisfied by $ u_n$. We also define
\[
 \overline{u}_n\colon [0,T]\times H_B\to \mathbb{R},\ (t,x)\mapsto \mathbb{E}
  \left[\overline h_n(\overline X^{t,x}_{n,T})\right],\  \ \ \ n\geq 1.
\]
We observe that $u_n=\overline{u}_{n|[0,T]\times H}$.

\begin{proposition}\label{prop:2015-06-21:02}
Let $p\geq 2$. Then:
\begin{enumerate}[(i)]
\item\label{2015-09-29:01}
$u_n$ is uniformly continuous on bounded sets of $[0,T]\times (H,|\cdot|_B)$ and, for every $t\in [0,T]$, $u_n(t,\cdot)$
is $|\cdot|_B$-Lipschitz continuous, with a Lipschitz constant uniform in $t\in [0,T]$ and $n\geq 1$.
\item\label{2015-09-29:02}  
  The sequence $\{u_{n}\}_{n\geq 1}$ converges to $u$ uniformly on compact sets of $[0,T]\times H$.
\item\label{2015-09-29:03}  For every $n\geq 1$, $u_n\in \mathcal{G}_s^{0,1}([0,T]\times H)$, and
  \begin{align}
&\sup_{\substack{(t,x)\in [0,T]\times H\\n\geq 1}}
|\nabla_x u_{n}(t,x)|_{H}<+\infty    \label{eq:2015-06-21:07},\\
&\sup_{\substack{(t,x)\in [0,T]\times H\\n\geq 1}}\left|A_n^*\nabla_x u_{n}(t,x)\right|_H<+\infty.
\label{2015-09-30:01}
  \end{align}
\end{enumerate}
\end{proposition}
\begin{proof}
\emph{(\ref{2015-09-29:01})} 
From \eqref{eq:supinfconvappr} and Proposition~\ref{prop:2015-06-20:03}-\emph{(\ref{2015-09-28:07})},\emph{(\ref{2015-09-28:09})},\emph{(\ref{2015-09-28:09b})}, it follows that $u_n$  is continuous and  $|\cdot|_B$-Lipschitz continuous in $x$ with a Lipschitz constant uniform in $t\in[0,T]$ and $n\geq 1$. 

The uniform continuity of $u_n$ on bounded sets is standard since we are dealing with bounded evolution and can be deduced from a more
general result, see e.g.\ \cite[Theorem~9.1]{DaPrato2014},
however we present a short argument.
We first notice that it follows
from
Proposition~\ref{prop:2015-06-20:03}-\emph{(\ref{2015-09-28:09})},\emph{(\ref{2015-09-28:09b})} that, for any $r>0$ and $n\geq 1$, there exists $K>0$ such that
$$
 \big|\overline X^{t,x}_n\big|_{\mathcal{H}^2_\mathcal{P}(H_B)}\leq K \qquad \forall t\in[0,T],\ \forall x\in H_B,\ |x|_B\leq r.
$$
Secondly, we recall that, for $t\in[0,T]$ and $x\in H_B$, $\overline X^{t,x}_n$ is a strong solution to
\eqref{eq:2015-06-21:00}, because $\overline A_n$ is bounded (see  footnote \ref{2016-05-18:00} on p.\ \pageref{2016-05-18:00}).
Then
if $0\leq t\leq t'\leq T$ and $x\in H_B$, $|x|_B\leq r$, for some constants $C_1$, $C_2$ depending only on $T$, $K$, $|\overline A_n|_{L(H_B)}$, and on the Lipschitz and the linear-growth constants of $\overline G_n$ and $\overline \Sigma_n$, by standard estimates we have
\begin{equation*}
  \mathbb{E} \left[ \left|\overline X^{t,x}_{n,s}-
\overline X^{t',x}_{n,s}
\right|_B ^2\right]\leq C_1 
(t'-t)+
C_2 
\int_{t'}^s\mathbb{E} \left[ \left|\overline X^{t,x}_{n,w}-\overline X^{t',x}_{n,w}\right|^2_B \right] dw\qquad \forall s\in [t',T].
\end{equation*}
By Gronwall's lemma, the inequality above provides
\begin{equation}
  \label{eq:2015-11-20:01}
  \mathbb{E} \left[ \left|\overline X^{t,x}_{n,T}-
\overline X^{t',x}_{n,T}
\right|_B ^2\right]\leq C_1e^{C_2T}(t'-t).
\end{equation}
The uniform continuity of $u_n$ on $[0,T]\times \{x\in H\colon |x|_B\leq r\}$ is then obtained by
\eqref{eq:supinfconvappr},
 \eqref{eq:2015-11-20:01}, and by the $|\cdot|_B$-Lipschitz continuity 
of $\overline X^{t,x}_n$ in $x$ 
with a Lipschitz constant uniform in $t\in[0,T]$.

\emph{(\ref{2015-09-29:02})}
Part \emph{(\ref{2015-09-29:02})} is a consequence of
Proposition \ref{prop:2015-06-20:03}-\emph{(\ref{2015-09-28:07})},\emph{(\ref{2015-09-28:08})}
 and (\ref{eq:supinfconvappr}).

\emph{(\ref{2015-09-29:03})}
Let $n\geq 1$. 
By \cite[Ch.\ 7, Proposition 7.3.3]{DaPrato2004}, the map
$$
\Xi_n:\mathcal{H}^p_\mathcal{P}(H_B)\rightarrow \mathcal{H}^p_\mathcal{P}(H_B),\ Z \mapsto \overline h_n(Z)
$$
belongs to $\mathcal{G}^1_s(\mathcal{H}^p_\mathcal{P}(H_B),\mathcal{H}^p_\mathcal{P}(H_B))$, and
\begin{equation}
  \label{eq:2015-09-29:05}
  (\nabla_Z\Xi_n(Z)Y)_t=D\overline h_n(Z_t)Y_t \ \ \ \ \forall t\in[0,T],\ \forall Z,Y\in \mathcal{H}^p_\mathcal{P}(H_B).
\end{equation}
 By 
Proposition \ref{prop:2015-06-20:03}-\emph{(\ref{2015-09-28:09})}, 
linearity and continuity of the expected valued $\mathbb{E}$ on $L^p_\mathcal{P}(H_B)$, linearity and continuity of 
the $T$-evaluation map
$\mathcal{H}^p_\mathcal{P}(H_B)\rightarrow L^p(H_B),\ Z \mapsto  Z_T$,
formula \eqref{eq:2015-09-29:05},
 composition 
of strongly continuously G\^ateaux differentiable functions,
 we obtain $\overline u_n\in \mathcal{G}^{0,1}_s([0,T]\times H_B)$ and
 \begin{equation}
\langle \nabla_x \overline u_n(t,x),y\rangle_B=\mathbb{E} \left[ D\overline h_n(\overline X^{t,x}_{n,T})
\left( \nabla_x\overline X_n^{t,x}y \right) _T \right]\ \ \ \ \ \forall (t,x,y)\in[0,T]\times H_B\times H_B.   
\label{eq:2015-09-29:04}
 \end{equation}
By 
Proposition \ref{prop:2015-06-20:03}-\emph{(\ref{2015-09-28:09b})},
\eqref{eq:supinfconvappr}, \eqref{eq:2015-09-29:04},
\begin{equation}
  \label{eq:2015-09-29:06}
   \sup_{\substack{(t,x)\in[0,T]\times H_B\\n\geq 1}}
|\nabla_x \overline u_{n}(t,x)|_B <+\infty.
\end{equation}
By continuous embedding $H\rightarrow H_B$ and by \eqref{eq:2015-09-29:06}
 we have also
\begin{equation}\label{2015-09-29:08}
u_n\in \mathcal{G}^{0,1}_s([0,T]\times H),\qquad\ \ \ \   \sup_{\substack{(t,x)\in[0,T]\times H\\n\geq 1}}
|\nabla_x  u_{n}(t,x)|_H <+\infty,
\end{equation}
which shows \eqref{eq:2015-06-21:07}. Moreover, since 
\[
  \nabla_xu_n(t,x)=(A^*-1)^{-1}(\overline A-1)^{-1}\nabla_x\overline u_n(t,x),
\]
we obtain from \eqref{eq:2015-09-29:06} that
\begin{equation}
  \label{eq:2015-09-29:09}
  \sup_{\substack{n\geq 1\\(t,x)\in[0,T]\times H}}|A^*\nabla_xu_n(t,x)|_H<+\infty.
\end{equation} 
Therefore, recalling that $S$ is a semigroup of contractions, we have 
\begin{equation*}
    |A^*_n\nabla_xu_n(t,x)|_H\leq
|n(n-A)^{-1}|_{L(H)}|A^*\nabla_xu_n(t,x)|_H
\leq 
|A^*\nabla_xu_n(t,x)|_H
\end{equation*}
for all $(t,x)\in [0,T] \times H$ which, together with \eqref{eq:2015-09-29:09}, shows 
\eqref{2015-09-30:01}.
\end{proof}

\vskip10pt
We now define for $n\geq 1$
$$
L_{n}\colon  [0,T]\times H\times H\times \mathbf{S}_m\rightarrow \mathbb{R},\ (t,x,p,P)\mapsto
\langle p,G_n(t,x)\rangle+\frac{1}{2}\operatorname{Tr}(\sigma_n(t,x)\sigma_n^*(t,x)P)
$$
where $\mathbf{S}_m$ is the set of $m\times m$ symmetric matrices.

We consider the following terminal value problems
\begin{equation}\label{eq:2015-01-16:01}
    \begin{dcases}
              -v_t-\langle A_nx,D_xv\rangle-L_n(t,x,D_{x}v,D_{x_0x_0}^2v)=0   &\quad (t,x)\in(0,T)\times H\\
        v(T,x)=h_n(x)&\quad x\in H.
      \end{dcases}
\end{equation} 
Since the operator $A_n$ is bounded we will use the definition of viscosity solution from \cite{Lions89jfa}.

\begin{definition}\label{def:2015-06-15:02}
    A locally bounded\footnote{By ``locally bounded'' we mean ``bounded on bounded subsets of the domain'', and by
     ``locally uniformly continuous'' we mean ``uniformly continuous on bounded subsets of the domain''.} upper semi-continuous function 
    $v$ on $(0,T]\times H$ is a viscosity
  subsolution of \eqref{eq:2015-01-16:01} if $v(T,x)\leq h_n(x)$ for all $x\in H$, and
  whenever $v-\varphi$ has a local maximum at a point $(\hat t,\hat x)\in
  (0,T)\times H$, for some $\varphi\in C^{1,2}((0,T)\times H)$, then
$$
-\varphi_t(\hat t,\hat x)-\langle A_n\hat x,D_x \varphi(\hat t,\hat
x)\rangle -L_n (\hat t,\hat x,D_x \varphi(\hat t,\hat
x),D^2_{x_0}\varphi(\hat t,\hat x))\leq 0.
$$
A locally bounded lower semi-continuous function $v$ on $(0,T]\times H$
is a viscosity
supersolution of \eqref{eq:2015-01-16:01} if $v(T,x)\geq h_n(x)$ for all $x\in H$, and 
whenever $v-\varphi$ has a local minimum at a point $(\hat t,\hat x)\in
(0,T)\times H$, for some 
$\varphi\in C^{1,2}((0,T)\times H)$,
then
$$
- \varphi_t(\hat t,\hat x)-\langle A_n\hat   x,D_x \varphi(\hat t,\hat
x)\rangle -L_n (\hat t,\hat x,D_x \varphi(\hat t,\hat
x),D^2_{x_0}\varphi(\hat t,\hat x))\geq 0.
$$ 
A viscosity solution of \eqref{eq:2015-01-16:01} is a function which is both a
viscosity subsolution and a viscosity supersolution of \eqref{eq:2015-01-16:01}.
\end{definition}

\begin{theorem}\label{propp:2014-02-15:ac}
For $n\geq 1$, the function $u_{n}$
is the unique (within the class of, say locally uniformly continuous functions with at most polynomial growth) viscosity solution of \eqref{eq:2015-01-16:01}.
\end{theorem}
\begin{proof}
Since $A_n$ is a bounded operator this is a standard result, see e.g.\ \cite{FGS, Kelome02, Lions89jfa}.
Notice that Proposition \ref{prop:2015-06-21:02}-\emph{(\ref{2015-09-29:01})} guarantees that the function $u_n$  
is locally uniformly continuous on $[0,T]\times H$ and is Lipschitz continuous in $x$.
\end{proof}

\begin{remark}\label{remBcontvisc}
This is not needed here however it is worth noticing that
 the function $u$ is the 
unique so called $B_{A, 1}$-continuous viscosity solution
(unique within the class of, say $B_{A, 1}$-continuous functions with at most polynomial growth which attain the terminal condition
locally uniformly), of the equation
\begin{equation}\label{eq:2015-01-16:01aaa}
    \begin{dcases}
              -u_t-\langle Ax,D_xu\rangle-L(t,x,D_{x}u,D_{x_0x_0}^2u)=0   &\quad (t,x)\in(0,T)\times H\\
        u(T,x)=h(x)&\quad x\in H,
      \end{dcases}
\end{equation}
where
\[
L\colon  [0,T]\times H\times H\times \mathbf{S}_m\rightarrow \mathbb{R},\ (t,x,p,P)\mapsto
\langle p,G(t,x)\rangle+\frac{1}{2}\operatorname{Tr}(\sigma(t,x)\sigma^*(t,x)P).
\]
For the proof of this we refer the reader to \cite[Theorem 3.64]{FGS}.

\end{remark}


\subsection{Space sections of viscosity solutions}


We skip the proof of the following basic lemma (for a very similar version, see \cite[Proposition~3.7]{Crandall1992}.
\begin{lemma}\label{lem:2012-11-15:aa}
  Let $D$ be a set, and $f,g\colon D\to \mathbb{R}$ be functions, with $g\geq
  0$.  Let 
$$
Z=\left\{y\in D\colon g(y)=0\right\}
$$ be the set of zeros of
  $g$.  Suppose that $Z\neq \emptyset$.  Let $\left\{h_i\colon D\to
    \mathbb{R}\right\}_{i\in \mathbb{N}}$ be a sequence of functions converging
  uniformly to $0$ in $D$ as $i\to +\infty$.  Let $\{\epsilon_i\}_{i\in \mathbb{N}}$ be
  a sequence of positive numbers decreasing to $0$.  Define
  $$
  \psi_i(y)\coloneqq f(y)-\frac{g(y)}{\epsilon_i}+h_i(y)\qquad \forall i\in \mathbb{N},\ \forall y\in D.
  $$
  Suppose that $\{y_i\}_{i\in \mathbb{N}}\subset D$ is a sequence such that
  $$
  \lim_{i\to \infty}\left[\sup_{y\in D}\psi_i(y)-\psi_i(y_i)\right]=0.
  $$
  Then ${\displaystyle  \lim_{i\to \infty} \frac{g(y_i)}{\epsilon_i}=0.  }$
\end{lemma}

Fix $\overline x_1\in H_1$.  Let $\phi\in C^{1,2}((0,T)\times
\mathbb{R}^m)$ and let $(\hat t,\hat x_0)\in (0,T)\times \mathbb{R}^m$ be a maximum
point of $u_n(\cdot,(\cdot,\overline x_1))-\phi(\cdot,\cdot)$ over $[0,T]\times\mathbb{R}^m$. Without loss of generality we can assume 
that $\phi\in C^{1,2}([0,T]\times
\mathbb{R}^m)$ and that the maximum is strict and global.

For $\epsilon>0$, define the function
\begin{equation}
  \label{eq:2015-06-18:02}
  \Phi_{\epsilon}(t,x_0,x_1)=\phi(t,x_0)+\frac{1}{\epsilon}| (0,x_1-\overline x_1)|_H^2,
\end{equation}
where $t\in (0,T)$, $(x_0,x_1)\in H$.  Observe that
$\Phi_{\epsilon} \in C^{1,2}([0,T]\times H)$,
 and
\begin{equation}
   \label{eq:2015-06-18:05}
\begin{split}
   &D_t\Phi_{\epsilon}(t,x)=\phi_t(t,x_0)\\
&D_x\Phi_{\epsilon}(t,x)=\left(D_{x_0}\phi(t,x_0),0\right)+\frac{2}{\epsilon}\left(0,x_1-\overline x_1\right)\\
& D^2_{x_0}\Phi_{\epsilon}(t,x)=D^2_{x_0}\phi(t,x_0).
\end{split}
\end{equation}

\begin{lemma}\label{lem:2012-11-20:aa}
For each $n\geq 1$,
 there exist real sequences $\{a_i\}_{i\in \mathbb{N}}$, $\{\epsilon_{i}\}_{i\in \mathbb{N}}$
     converging to $0$, and a sequence $\{p_i\}_{i\in \mathbb{N}}$ converging to the origin in $H$, such
     that the function
     \begin{equation}
    (0,T)\times H\to \mathbb{R},\
        (t,x) \mapsto  u_{n}(t,x)-\Phi_{\epsilon_i}(t,x)-\langle p_i,x\rangle -a_it
     \end{equation}
     has a strict global maximum at $(t_i, x_i)$ and the
     sequence $\{( t_i,x_i)\}_{i\in \mathbb{N}}$ converges to $(\hat t,(\hat
     x_0,\overline x_1))$.
 \end{lemma}
 \begin{proof}
Let $R>|(\hat x_0,\overline x_1)|_H$
and $\mathbf{B}_R\coloneqq \{x\in H\colon |x|_H\leq R\}$.
 Let $\{\epsilon_{i}\}_{i\in \mathbb{N}}$ be a sequence converging to $0$. Applying the classical result of
Ekeland and Lebourg   \cite{EkelandLebourg76, Stegall78},
there exist sequences $\{a_i\}_{i\in \mathbb{N}}\subset \mathbb{R}$ and $\{p_i\}_{i\in \mathbb{N}}\subset  H$ such that $|a_i|\leq 1/i$, $|p_i|_H\leq 1/i$, and such that the function
\[
[0,T]\times \mathbf{B}_R\rightarrow \mathbb{R},\ 
u_{n}(t,x)-\Phi_{\epsilon_i}(t,x)-\langle p_i,x\rangle -a_it
\]
has a strict global maximum
at some point  $(t_i, x_i)\in [0,T]\times \mathbf{B}_R$.
By applying
Lemma~\ref{lem:2012-11-15:aa}
with $D=[0,T]\times \mathbf{B}_R$, $f(t,x)=u_n(t,x)-\varphi(t,x_0)$, $g(t,x)=|(0,x_1-\overline{x}_1)|_H^2$, $h_i(t,x)=-\langle p_i,x\rangle-a_it$, $y_i=(t_i,x_i)$,
we obtain
 \begin{equation}
   \label{eq:2015-06-18:00}
  \lim_{i\rightarrow \infty}|
   (0,x_{i,1}-\overline x_1)|_H= 0.
 \end{equation}

To conclude the proof it is then sufficient to 
 show that $(t_i,x_{i,0})\to (\hat
   t,\hat x_0)$.
 Indeed, suppose that this does not hold.
 Up to a subsequence, we can suppose that $(t_i,x_{i,0})\to
   (\widetilde t,\widetilde x_0)\neq (\hat t,\hat x_0)$.  Since, by assumption,
   $(\hat t,\hat x_0)$ is a strict global maximum point of
   $u_{n}(\cdot,(\cdot,\overline x_1))-\phi(\cdot,\cdot)$, there exists
   $\eta>0$ such that, for $i$ sufficiently large, we have
   \begin{equation}\label{eq:2012-11-20:ad}
     \begin{split}
            u_{n}(\hat t,(\hat x_0,\overline x_1))-\phi(\hat t,\hat x_0)\geq& \eta+ u_n(t_i,(x_{i,0},\overline x_1))-\varphi(t_i,x_{i,0})\\
 \geq &   \eta+u_{n}(t_i,(x_{i,0},\overline x_1))-
     \Phi_{\epsilon_i}(t_i,x_i)\\
     = &\eta+ \left(u_{n}(t_i,(x_{i,0},\overline x_1))- u_{n}(t_i,x_i)\right)
     +u_{n}(t_i,x_i) - \Phi_{\epsilon_i}(t_i,x_i) 
     \\
  \geq &\eta+ \left(u_{n}(t_i,(x_{i,0},\overline x_1))- u_{n}(t_i,x_i)\right)
     +u_{n}(\hat t,(\hat x_{0},\overline x_1))-\phi(\hat t,\hat x_{0})\\
&     +\langle p_i,x_i-(\hat x_{0},\overline x_1)\rangle+a_i(t_i-\hat t).
   \end{split}
 \end{equation}
By \eqref{eq:2015-06-18:00},
${\displaystyle \lim_{i\rightarrow \infty}
(x_{i,0},x_{i,1})=(\widetilde x_0,\overline x_1).}
$
Thus by  continuity of $u_{n}$, for $i$ sufficiently large, we have
  $$
  |u_{n}(t_i,(x_{i,0},\overline x_1))- u_{n}(t_i,x_i)| \leq \frac{\eta}{2}
  $$
and then if follows from \ref{eq:2012-11-20:ad} that
  $$
  u_{n}(\hat t,(\hat x_0,\overline x_1))-\phi(\hat t,\hat x_0)\geq
  \frac{\eta}{2}+ u_{n}(\hat t,(\hat x_0,\overline x_1))-\phi(\hat t,\hat
  x_0) +\langle p_i,x_i-(\hat x_{0},\overline
  x_1)\rangle+a_i(t_i-\hat t).
  $$
  This produces a contradiction by letting $i\to+\infty$, recalling
  that $p_i\to 0$ and $a_i\to 0$.  Thus we must have ${\displaystyle \lim_{i\rightarrow \infty}(t_i,x_{i,0})=(\hat t,\hat x_0)}$.  
\end{proof}

\vskip10pt

For any $\overline x_1\in H_1$ and $n\in \mathbb{N}$, we define  the following functions
\begin{equation}
  \label{eq:2015-06-23:00}
  v_{n,\overline x_1}\colon [0,T]\times \mathbb{R}^m\rightarrow \mathbb{R},\, (t,x_0)\mapsto v_{n,\overline x_1}(t,x_0)\coloneqq u_{n}(t,(x_0, \overline x_1)),
\end{equation}
\begin{equation}
  \label{eq:2015-06-24:00}
  a_{n,\overline x_1}\colon [0,T]\times \mathbb{R}^m\rightarrow \mathbf{S}_m,\,
(t,x_0)\mapsto  \sigma_n(t,(x_0,\overline x_1))\sigma_n^*(t,(x_0,\overline x_1))
\end{equation}
and
\begin{equation}
  \label{eq:2015-06-23:02}
  \beta_{n,\overline x_1}\colon [0,T]\times \mathbb{R}^m\rightarrow \mathbb{R},\, (t,x_0)\mapsto 
\langle A_n(x_0,
  \overline x_1)+G_n(t,(x_0,\overline x_1)),\nabla_xu_n(t,(x_0,\overline x_1))\rangle.
\end{equation}
We  associate to \eqref{eq:2015-01-16:01} the following terminal value problem
  \begin{equation} \label{eq:2015-06-18:01}
    \begin{dcases}
                -v_t(t,x_0)-\frac{1}{2}\operatorname{Tr}(
a_{n,\overline x_1}(t,x_0)D^2_{x_0}v(t,x_0)
)
-\beta_{n,\overline x_1}(t,x_0)=0
          & (t,x_0)\in(0,T)\times\mathbb{R}^m\\
          v(T,x_0)=h_n(x_0,\overline x_1)& x_0\in \mathbb{R}^m.
        \end{dcases}
      \end{equation}

We recall that it follows from 
Proposition \ref{prop:2015-06-21:02}-\emph{(\ref{2015-09-29:03})} that
for every $\overline x_1\in H_1$ the function $\beta_{n,\overline x_1}$ is continuous and for every compact set $K\subset\mathbb{R}^m$,
\begin{equation}
  \label{eq:2015-06-23:0211}
  \sup_{\substack{n\ge 1\\(t,x_0)\in [0,T]\times K}} |\beta_{n,\overline x_1}(t,x_0)|< +\infty.
\end{equation}

In the following proposition we show that the section functions $v_{n,\overline x_1}$ are the viscosity solutions of \eqref{eq:2015-06-18:01}.
For the definition of viscosity solution in finite dimensions, we refer to \cite{Crandall1992}.

\begin{proposition}\label{propp:2013-01-13-ab}
For every $\overline x_1\in H_1$ and $n\geq 1$,
 $v_{n,\overline{x}_1}$
  is a viscosity solution  of
\eqref{eq:2015-06-18:01}.
\end{proposition}
\begin{proof}
We prove that $v_{n,\overline x_1}$ is a subsolution. The supersolution case is similar.
The continuity of $u_{n}$ implies the 
  continuity of $v_{n,\overline x_1}$.  Let $\phi\in
  C^{1,2}((0,T)\times \mathbb{R}^m)$ 
be such that
$v_{n,\overline    x_1}-\phi$ has a local maximum at $(\hat t,\hat x_0)\in
  (0,T)\times \mathbb{R}^m$. Without loss of generality, we can assume that the maximum is strict and global and that $\phi\in
  C^{1,2}([0,T]\times \mathbb{R}^m)$. By Lemma \ref{lem:2012-11-20:aa}, there exist real sequences
  $\{\epsilon_i\}_{i\in \mathbb{N}}, \{a_i\}_{i\in \mathbb{N}}$ converging to $0$, and a sequence
  $\{p_i\}_{i\in \mathbb{N}}$ in $H$ converging to $0$, such that the functions
$$
[0,T]\times \mathbb{R}^m\rightarrow \mathbb{R},\,  (t,x)\mapsto u_{n}(t,x)-\Phi_{\epsilon_i}(t,x)-\langle p_i,  x\rangle-a_it
$$
 have local maxima at $( t_i, x_i)$ and the sequence $\{(t_i,x_i)\}_{i\in \mathbb{N}}$ converges to $(\hat t,(\hat x_0,\overline x_1))$.  
  Since $u_{n}$ is a viscosity solution of \eqref{eq:2015-01-16:01}, we have
  \begin{equation}
    \label{eq:2015-06-18:03}
    \begin{multlined}[c][0.85\displaywidth]
          -D_t\Phi_{\epsilon_i}\left(t_i, x_i \right)
-a_i
-\langle A_n x_i, D_x\Phi_{\epsilon_i}(t_i, x_i)+p_i\rangle
\\ -L_n\left(t_i, x_i,D_{x}\Phi_{\epsilon_i}(t_i, x_i )+p,D^2_{x_0}\Phi_{\epsilon_i}(t_i, x_i) \right)\leq 0.
\end{multlined}
\end{equation}
   Since $u_{n}\in \mathcal{G}_s^{0,1}([0,T]\times H,\mathbb{R})$, we must have
  \begin{equation}\label{eq:2012-11-22:ab}
    \nabla_xu_{n}(t_i, x_i)=D_x\Phi_{\epsilon_i}(t_i, x_i)+p_i.
  \end{equation}    
Thus, by recalling \eqref{eq:2015-06-18:05}, we have
  \begin{equation}
      -D_t\Phi_{\epsilon_i}\left(t_i, x_i \right)
-a_i
-\langle A_n x_i,\nabla_xu_{n}(t_i, x_i)\rangle
 -L_n\left(t_i, x_i,\nabla_xu_{n}(t_i, x_i)
,D^2_{x_0}\phi(t_i, x_{i,0}) \right)\leq 0.
\end{equation}
We now pass to the limit $i\rightarrow +\infty$ and, by \eqref{eq:2015-06-18:05} and the strong continuity of $\nabla_x u_n$, we obtain
\begin{equation*}
  -\phi_t\left(\hat t, \hat x_0 \right)
  -\langle A_n (\hat x_0,\overline x_1),\nabla_xu_{n}(\hat t, (\hat x_0, \overline x_1))\rangle
  -L_n\left(\hat t, (\hat x_0,\overline x_1),\nabla_xu_{n}(\hat t, (\hat x_0,\overline x_1))
    ,D^2_{x_0}\phi(\hat t, \hat x_0) \right)\leq 0,
\end{equation*}
which can be written, by 
using the definition of $\beta_{n,\overline x_1}$,
$$
-\phi_t(\hat t,\hat x_0)
-\frac{1}{2}\operatorname{Tr}\left(   
 \left(D^2_{x_0}\phi(\hat t,\hat x_0)\right)
 \sigma_n(\hat t,(\hat x_0,\overline x_1))\sigma_n^*(\hat t,(\hat x_0, \overline x_1))\right)
-\beta_{n,\overline x_1}(\hat t,\hat x_0)
\leq
0.
$$
Thus $v_{n,\overline x_1}$ is a viscosity subsolution of \eqref{eq:2015-06-18:01}.
  \end{proof}

\subsection{Regularity with respect to the finite dimensional component}

In this section we show that, 
if $\sigma$ is non-degenerate, then the function $u$ defined by \eqref{eq:2015-01-16:00} is differentiable with respect to $x_0$ and $D_{x_0}u$ enjoys some H\"older continuity.

\begin{theorem}
Suppose that, for every $(t,x)\in [0,T]\times H$ and $y\in \mathbb{R}^m$, $\sigma(t,x)y\neq 0$.
Then, for every $\overline x_1\in H_1$, the function $v_{\overline x_1}$ defined by $v_{\overline x_1}(t,x_0)\coloneqq u(t,(x_0,\overline x_1))$
belongs to $C^{\alpha+1}_{loc}((0,T)\times \mathbb{R}^m)$, for every $\alpha\in (0,1)$.
\end{theorem}
\begin{proof}
Let $(t,x_0)\in (0,T)\times \mathbb{R}^m$. Let $Q\coloneqq [c,d)\times B(x_0,\epsilon)$ be a neighborhood of $(t,x_0)$ in $(0,T)\times \mathbb{R}^m$ such that, for some $M>0$ and $\delta>0$, $\delta<a_{\overline x_1}(s,y)\coloneqq \sigma(s,(y,\overline x_1))\sigma^*(s,(y,\overline x_1))<M$ for all 
$(s,y)\in Q$.
Since $\Sigma_n(s,(y,\overline x_1))z=(\sigma_n(s,(y,\overline x_1))z,0_1)$ and $\{\sigma_n\}_{n\in \mathbb{N}}$ converges to $\sigma$ uniformly on compact sets (Remark \ref{2015-09-30:00}), we can suppose that $\delta<a_{n,\overline x_1}(s,y)<M$ for all $n\in \mathbb{N}$ and $(s,y)\in Q$ and that the family  $\{a_{n,\overline x_1}\}_{n\in \mathbb{N}}$ 
is equi-uniformly continuous.

By Proposition \ref{propp:2013-01-13-ab}, for $n\geq 1$, $v_{n,\overline x_1}$ is a viscosity solution of \eqref{eq:2015-06-18:01}, in particular
it is a viscosity solution of the terminal boundary  value problem
  \begin{equation} \label{eq:2015-06-23:07}
    \begin{dcases}
                -v_t(s,y)-\frac{1}{2}\operatorname{Tr}(a_{n,\overline x_1}(s,y)D^2_{y}v(s,y))
-\beta_{n,\overline x_1}(s,y)=0
          & (s,y)\in Q\\
          v(s,y)=u_n(s,(y,\overline x_1))& (s,y)\in \partial_PQ
        \end{dcases}
      \end{equation}
Thus, for instance by \cite[Lemma~2.9, Proposition 2.10, and Theorem 9.1]{Crandall2000}, $v_{n,\overline x_1}$ is the unique viscosity
solution (in particular also a unique $\leb^p$-viscosity solution\footnote{See \cite{Crandall2000} for the definition of $L^p$-viscosity solution.}) 
of \eqref{eq:2015-06-23:07}, and
\begin{equation}
  \label{eq:2015-06-23:10}
  |v_{n,\overline x_1}|_{W^{1,2,p}(Q')}\leq C\left(\sup_{(s,y)\in Q}|u_n(s,(y,\overline x_1))|+\sup_{(s,y)\in Q}|\beta_n(s,(y,\overline x_1))|\right)
\end{equation}
for all $m+1\leq p<+\infty$ and for all $Q'=[c',d')\times B(x,\epsilon')$, with $c<c'<d'<d$ and $0<\epsilon'<\epsilon$, and where $C$ depends only on $m$,  
$p$,  $\delta$, $M$,  $Q$, $Q'$, and the uniform modulus of continuity of the functions $a_{n,\overline x_1}$. Thus, by 
Proposition \ref{prop:2015-06-21:02} and \eqref{eq:2015-06-23:0211}, 
the set $\{v_{n,\overline x_1}\}_{n\geq 1}$ is uniformly bounded in $W^{1,2,p}(Q')$. Therefore applying an embedding theorem, 
see e.g.\ 
\cite[Lemma~3.3, p.\ 80]{Ladyvzenskaja1968}, we obtain that for every $\alpha\in(0,1)$
\[
|v_{n,\overline x_1}|_{C^{1+\alpha}(Q')}\leq C_\alpha
\]
for some constant $C_\alpha$ independent of $n$. Since the sequence $\{v_{n,\overline x_1}\}_{n\geq 1}$ converges uniformly 
on compact sets to
the function $v_{\overline x_1}$ as $n\to+\infty$,
it follows that the function $v_{\overline x_1}$ satisfies the above estimate too.
This completes the proof.
\end{proof}

\bibliographystyle{plain}
\bibliography{RoSw_2018-06-20_arxiv.bbl}

\begin{thebibliography}{10}

\bibitem{Bjork2003}
T.~Bj\"ork.
\newblock {\em Arbitrage theory in continous time}.
\newblock Oxford University Press, 2$^{\mathrm{nd}}$ edition, 2003.

\bibitem{Chojnowska-Michalik1978}
A.~Chojnowska-Michalik.
\newblock Representation theorem for general stochastic delay equations.
\newblock {\em Bulletin de l'Academie Polonaise des Sciences. S\'erie des
  sciences math., astr. et phys.}, XXVI(7):635--642, 1978.

\bibitem{Crandall1992}
M.G. Crandall, H.~Ishii, and P.L. Lions.
\newblock User's guide to viscosity solutions of second order partial
  differential equations.
\newblock {\em Bull. Amer. Math. Soc.}, 27:1--67, 1992.

\bibitem{Crandall2000}
M.G. Crandall, M.~Kocan, and A.~\Swiech.
\newblock ${L}^p$-theory for fully nonlinear uniformly parabolic equations.
\newblock {\em Comm. Partial Differential Equations}, 25:1997--2053, 2000.

\bibitem{Crandall1990237}
M.G. Crandall and P.L. Lions.
\newblock Viscosity solutions of {H}amilton-{J}acobi equations in infinite
  dimensions. {IV}. {H}amiltonians with unbounded linear terms.
\newblock {\em Journal of Functional Analysis}, 90(2):237--283, 1990.

\bibitem{CRANDALL1991417}
M.G. Crandall and P.L. Lions.
\newblock Viscosity solutions of {H}amilton-{J}acobi equations in infinite
  dimensions. {V}\@. {U}nbounded linear terms and {$B$}-continuous solutions.
\newblock {\em Journal of Functional Analysis}, 97(2):417--465, 1991.

\bibitem{DaPrato2014}
G.~Da~Prato and J.~Zabczyck.
\newblock {\em Stochastic Equations in Infinite Dimensions}.
\newblock Cambridge University Press, $2^{\textrm{nd}}$ edition, 2014.

\bibitem{DaPrato2004}
G.~Da~Prato and J.~Zabczyk.
\newblock {\em Second Order Partial Differential Equations in Hilbert Spaces}.
\newblock Cambridge University Press, Cambridge, 2002.

\bibitem{EkelandLebourg76}
I.~Ekeland and G.~Lebourg.
\newblock Generic {F}r\'echet-differentiability and perturbed optimization
  problems in {B}anach spaces.
\newblock {\em Trans. Amer. Math. Soc.}, 224(2):193--216, 1976.

\bibitem{Engel2000}
K.-J. Engel and R.~Nagel.
\newblock {\em One-parameter semigroups for linear evolution equations}.
\newblock Springer, 2000.

\bibitem{FGS}
G.~Fabbri, F.~Gozzi, and A.~\Swiech.
\newblock {\em Stochastic Optimal Control in Infinite Dimensions: Dynamic
  Programming and HJB Equations, with Chapter 6 by M. Fuhrman and G.
  Tessitore}.
\newblock Book in preparation. Chapters 1-3 are available at
  http://people.math.gatech.edu/$\sim$swiech/FGS-Chapters1-3.pdf.

\bibitem{FedGolGoz2010}
S.~Federico, B.~Goldys, and F.~Gozzi.
\newblock {HJB} equations for the optimal control of differential equations
  with delays and state constraints, {I}: regularity of viscosity solutions.
\newblock {\em SIAM J. Control Optim.}, 48(8):4910--4937, 2010.

\bibitem{Granas2003}
A.~Granas and J.~Dugundji.
\newblock {\em Fixed Point Theory}.
\newblock Springer, 2003.

\bibitem{Kelome02}
D.~Kelome.
\newblock {\em Viscosity solution of second order equations in a separable
  {H}ilbert space and applications to stochastic optimal control}.
\newblock PhD thesis, Georgia Institute of Technology, 2002.

\bibitem{KocanSwiech95}
M.~Kocan and A.~\Swiech.
\newblock Second order unbounded parabolic equations in separated form.
\newblock {\em Studia Math.}, 115:291--310, 1995.

\bibitem{Ladyvzenskaja1968}
O.A. Lady\v{z}enskaja, V.A. Solonnikov, and N.N. Ural'ceva.
\newblock {\em Linear and Quasi-linear Equations of Parabolic Type}.
\newblock American Mathematical Society, 1968.

\bibitem{LL}
J.-M. Lasry and P.-L. Lions.
\newblock A remark on regularization in {H}ilbert spaces.
\newblock {\em Israel J. Math.}, 55(3):257--266, 1986.

\bibitem{Lions88}
P.-L. Lions.
\newblock Viscosity solutions of fully nonlinear second-order equations and
  optimal stochastic control in infinite dimensions. {I}. {T}he case of bounded
  stochastic evolutions.
\newblock {\em Acta Math.}, 161(3-4):243--278, 1988.

\bibitem{Lions89jfa}
P.-L. Lions.
\newblock Viscosity solutions of fully nonlinear second-order equations and
  optimal stochastic control in infinite dimensions. {III}. {U}niqueness of
  viscosity solutions for general second-order equations.
\newblock {\em J. Funct. Anal.}, 86(1):1--18, 1989.

\bibitem{peszat1995}
S.~Peszat and J.~Zabczyk.
\newblock Strong {F}eller {P}roperty and {I}rreducibility for {D}iffusions on
  {H}ilbert {S}paces.
\newblock {\em The Annals of Probability}, 23(1):157--172, 1995.

\bibitem{Stegall78}
C.~Stegall.
\newblock Optimization of functions on certain subsets of {B}anach spaces.
\newblock {\em Math. Ann.}, 236(2):171--176, 1978.

\bibitem{SwiechTeixeira09}
A.~\Swiech$\, $ and E.V. Teixeira.
\newblock Regularity for obstacle problems in infinite dimensional {H}ilbert
  spaces.
\newblock {\em Adv. Math.}, 220(3):964--983, 2009.

\bibitem{'Swiech1994}
A.~\Swiech.
\newblock {"Unbounded" Second Order Partial Differential Equations in Infinite
  Dimensional Hilbert Spaces}.
\newblock {\em Comm. Partial Differential Equations}, 19(11-12):1999--2036,
  1994.

\end{thebibliography}

\end{document}